\documentclass[11pt,a4paper,reqno]{amsart}
\usepackage[utf8]{inputenc}
\usepackage[british]{babel}

\usepackage[svgnames]{xcolor}

\usepackage[linktoc=all,colorlinks=true,linkcolor=Blue,citecolor=Green,urlcolor=Purple]{hyperref}
\hypersetup{pdftitle={Strong smoothing for the non-cutoff homogeneous Boltzmann equation for Maxwellian molecules with Debye-Yukawa type interaction}, pdfauthor={Barbaroux, Hundertmark, Ried, Vugalter}, pdfsubject={35D10 (Primary); 35B65, 35Q20, 82B40 (Secondary)}, pdfkeywords={Smoothing of weak solutions, Non-cutoff homogeneous Boltzmann equation, Debye-Yukawa potential, Maxwellian molecules}} 

\usepackage{amsmath,amssymb,amsthm,amsfonts}
\usepackage{dsfont}
\usepackage[margin={3cm}]{geometry}
\usepackage[full]{textcomp}
\usepackage[T1]{fontenc}
\usepackage{txfonts}
\usepackage[cal=boondoxo]{mathalfa}
\usepackage{microtype}
\usepackage[inline]{enumitem}
\usepackage{tikz}
\usetikzlibrary{arrows}
\newcommand{\R}{\mathbb{R}}
\newcommand{\C}{\mathbb{C}}
\newcommand{\N}{\mathbb{N}}

\renewcommand{\S}{\mathbb{S}}
\newcommand{\1}{\mathds{1}}
\newcommand{\I}{\mathrm{i}}
\newcommand{\E}{\mathrm{e}}
\newcommand{\D}{\mathrm{d}}

\renewcommand{\Re}{\mathrm{Re}\,}
\renewcommand{\coloneq}{~\raise.35pt\hbox{:}\kern-3.5pt=}
\renewcommand{\eqcolon}{=\kern-3.5pt\raise.35pt\hbox{:}~}

\DeclareMathOperator{\Hyp}{Hyp}

\newcommand{\etal}{\textsl{et al.\ }}

\theoremstyle{plain}
\newtheorem{proposition}{Proposition}[section]
\newtheorem{corollary}[proposition]{Corollary}
\newtheorem{theorem}[proposition]{Theorem}
\newtheorem*{theorem*}{Theorem}
\newtheorem{lemma}[proposition]{Lemma}
\newtheorem*{conjecture*}{Conjecture}

\theoremstyle{definition}
\newtheorem{definition}[proposition]{Definition}
\newtheorem{remark}[proposition]{Remark}
\newtheorem{remarks}[proposition]{Remarks}
\newtheorem*{remarks*}{Remarks}

\newcounter{smallenum}
\newenvironment{SE}{\begin{list}{{(\rm\arabic{smallenum})}}{%
\setlength{\topsep}{0mm}\setlength{\parsep}{0mm}\setlength{\itemsep}{0mm}%
\setlength{\labelwidth}{2em}\setlength{\leftmargin}{1.6em}\usecounter{smallenum}%
}}{\end{list}}
\begin{document}
\title[Strong smoothing for the homogeneous Boltzmann Equation]{Strong smoothing for the non-cutoff homogeneous Boltzmann equation for Maxwellian molecules with Debye-Yukawa type interaction}

\author{Jean-Marie Barbaroux}

\author{Dirk Hundertmark}

\author{Tobias Ried}

\author{Semjon Vugalter}

\date{16th December 2015}
\thanks{\copyright~2015 by the authors. Faithful reproduction of this article, in its entirety, by any means is permitted for non-commercial purposes}
\begin{abstract}
We study weak solutions of the homogeneous Boltzmann equation for Maxwellian molecules with a logarithmic singularity of the collision kernel for grazing collisions. Even though in this situation the Boltzmann operator enjoys only a very weak coercivity estimate, it still leads to strong smoothing of weak solutions in accordance to the smoothing expected by an analogy with a logarithmic heat equation.  
\end{abstract}
\maketitle
{\noindent
2010 \textit{Mathematics Subject Classification}: 35D10 (Primary); 
 35B65, 35Q20, 82B40 (Secondary) \\ 
\textit{Keywords}: Smoothing of weak solutions, Non-cutoff homogeneous Boltzmann equation, Debye-Yukawa potential, Maxwellian molecules
}
{\hypersetup{linkcolor=black}
\tableofcontents}
\section{Introduction and main results}\label{sec:introduction}

We study the regularity of weak solutions of the Cauchy problem
\begin{align}\label{eq:cauchyproblem}	
\begin{cases}
\partial_t f = Q(f,f) & \\
f|_{t=0} = f_0 &
\end{cases}
\end{align}
for the fully nonlinear homogeneous Boltzmann equation in $d\geq 2$ dimensions 
with initial datum $f_0 \geq 0$ having finite mass, energy and entropy, $f_0\in L^1_2(\R^d) \cap L\log L(\R^d)$.

The bilinear Boltzmann collision operator $Q$ is given by
\begin{align}\label{eq:Boltzmann-kernel}
	Q(g,f) = \int_{\R^d} \int_{\S^{d-1}} B\left(|v-v_*|,\frac{v-v_*}{|v-v_*|} \cdot \sigma\right) \left( g(v'_*) f(v') - g(v_*) f(v) \right) \, \mathrm{d}\sigma \mathrm{d} v_*.
\end{align}
Here we use the $\sigma$-representation of the collision process, in which
\begin{align*}
	v' = \frac{v+v_*}{2} + \frac{|v-v_*|}{2}\sigma, \quad v'_* = \frac{v+v_*}{2} - \frac{|v-v_*|}{2}\sigma, \quad \text{for } \sigma \in \S^{d-1}.
\end{align*}

The collision kernel $B$ takes into account the detailed scattering process by which the particles change their velocities, which, in a dilute gas, can be assumed to involve only two particles at a time (\emph{binary} collisions). In the important case of inverse-power-law interactions $\Phi(r) = r^{1-n}$, $n>2$, it is of the form  $B(|v-v_*|, \cos\theta) = |v-v_*|^{\gamma} b(\cos\theta)$, $\gamma=\frac{n-(2d-1)}{n-1}$, where $\cos\theta = \frac{v-v_*}{|v-v_*|} \cdot \sigma$ and $b$ is the so called \emph{angular collision kernel}. Even though an explicit formula for  $b$ is not known, one can show \cite{Cer88} that $b$ is smooth away from the singularity, non-negative, and has the non-integrable singularity
	\begin{align}\label{eq:powerlaw}
		\sin^{d-2}\theta \, b(\cos\theta) \stackrel{\theta\to 0}{\sim} \frac{K}{\theta^{1+2\nu}}
	\end{align}
	for some $K>0$ and $0<\nu<1$. 
	
It has been noted for some time now, see \cite{ADVW00} and the references therein, that the divergence \eqref{eq:powerlaw} leads to a coercivity in the Boltzmann collision kernel of the form 
\begin{align}\label{eq:intuition1}
	-Q(g,f) \approx (-\Delta)^{\nu} f + \text{lower order terms}, 
\end{align}
that is, it behaves similar to a singular integral operator with leading term proportional to a fractional Laplacian. 
If the interaction is instead of Debye-Yukawa type
\begin{align}\label{eq:debyeyukawa}
	\Phi(r) = r^{-1} \E^{-r^{s}}, \quad 0<s<2.
\end{align} 
the angular collision cross-section $b(\cos\theta)$ has a much weaker non-integrable singularity of logarithmic type 
\begin{align}\label{eq:singularity}
	\sin^{d-2}\theta \, b(\cos\theta) \sim \kappa \theta^{-1} \left(\log\theta^{-1}\right)^{\mu}
\end{align}
for grazing collisions $\theta\to 0$, with some $\kappa, \mu >0$. For example in dimension $d=3$ one has $\mu=\frac{2}{s}-1$ \cite{MUXY09}. Going through the calculations of \cite{MUXY09} one can check that $\mu=\frac{d-2}{s} -1$ in arbitrary dimension $d\geq2$.
In this case, the coercive effects are much weaker and of the form 
\begin{align}\label{eq:intuition2}
	-Q(g,f) \approx (\log(1-\Delta))^{\mu+1} f + \text{lower order terms},
\end{align}
as was noticed in \cite{MUXY09}, see also appendix \ref{app:coercivity}. 
In this work, as in \cite{BHRV15,MUXY09}, we consider only the so-called Maxwellian molecules approximation, where the collision kernel does not depend on $v-v_*$, but only on the collision angle $\theta$.

Even though $b$ has a singularity, the quantity 
\begin{align}\label{eq:momentumtransfer}
	\int_{0}^{\frac{\pi}{2}} \sin^d \theta \,b(\cos\theta)\,\D\theta <\infty,
\end{align}
which is related to the momentum transfer in the scattering process, is finite in both cases \eqref{eq:powerlaw} and \eqref{eq:debyeyukawa}.

We will also assume that $b(\cos\theta)$ is supported on angles $\theta \in [0, \frac{\pi}{2}]$, which is always possible due to symmetry properties of the Boltzmann collision operator.
 
In this article, we use the following notations and conventions: Given a vector $v\in\R^d$ and $\alpha\geq0$, let  $\langle v \rangle_{\alpha} \coloneq  (\alpha+ |v|^2)^{1/2}$, and $\langle v\rangle \coloneq  \langle v \rangle_1$. For $p\geq 1$ and $s\in\R$ the weighted $L^p$ spaces are given by
\begin{align*}
	L^p_{s}(\R^d) \coloneq  \left\{ f\in L^p(\R^d): \langle \cdot \rangle^{s} f \in L^p(\R^d) \right\},
\end{align*}
equipped with the norm 
\begin{align*}
	\| f\|_{L^p_s(\R^d)} = \left( \int_{\R^d} |f(v)|^p \langle v \rangle^{s p} \,\mathrm{d}v \right)^{1/p}.
\end{align*}
We will also make use of the weighted ($L^2$-based) Sobolev spaces 
\begin{align*}
	H^k_{\ell}(\R^d) = \left\{ f\in \mathcal{S}'(\R^d): \langle \cdot \rangle^{\ell} f \in H^k(\R^d) \right\}, \quad k,\ell \in \R,
\end{align*}
where $H^k(\R^d)$ are the usual Sobolev spaces
  given by 	$H^k (\R^d) = \left\{ f\in \mathcal{S}'(\R^d): \langle \cdot \rangle^{k} \hat{f} \in L^2(\R^d) \right\}$, for $k  \in \R$. We also use $H^{\infty}(\R^d) = \bigcap_{k\geq 0} H^{k}(\R^d)$. 
The inner product on $L^2(\R^d)$ is given by $\langle f,g\rangle = \int_{\R^d} \overline{f(v)} g(v) \,\mathrm{d}v$.
Further, closely related to the functions with finite (negative) entropy $\mathcal{H}(f)\coloneq \int_{\R^d} f\log f\, \mathrm{d}v$ is the space
\begin{align*}
	L\log L(\R^d) = \left\{f:\R^d\to\R \text{ measurable}: \|f\|_{L\log L} = \int_{\R^d} |f(v)| \log\left(1+|f(v)|\right)\,\mathrm{d}v < \infty\right\}.
\end{align*}

We use the following convention regarding the Fourier transform of a function $f$ in this article,
\begin{align*}
	(\mathcal{F} f) (\eta) = \hat{f}(\eta) = \int_{\R^d} f(v) \, \E^{-2\pi\I v\cdot \eta} \,\D v.
\end{align*}
We denote $D_v=-\tfrac{\I}{2\pi}\nabla$ and for a suitable function $G:\R^d \to \C$ we define the operator $G(D_v)$ as a Fourier multiplier, that is, 
\begin{align*}
	G(D_v)f\coloneq  \mathcal{F}^{-1}[ G\hat{f}].
\end{align*}
The precise notion of a solution of the Cauchy problem \eqref{eq:cauchyproblem} is given by
\begin{definition}[Weak Solutions of the Cauchy Problem \eqref{eq:cauchyproblem} \cite{Ark81,Vil98}]\label{def:weaksol}
	Assume that the initial datum $f_0$ is in $L^1_2(\R^d)\cap L\log L(\R^d)$. $f: \R_+ \times \R^d \to \R$ is called a weak solution to the Cauchy problem \eqref{eq:cauchyproblem}, if it satisfies the following conditions\footnote{Throughout the text, whenever not explicitly mentioned, we will drop the dependence on $t$ of a function, i.e. $f(v)\coloneq  f(t,v)$ etc}:
	\begin{enumerate}[label=(\roman*)]
		\item $f\geq 0$,\,\,\,  $f\in\mathcal{C}(\R_+; \mathcal{D}'(\R^d))\cap  L^{\infty}(\R_+;L^1_2(\R^d)\cap L \log L(\R^d))$  
		\item $f(0,\cdot) = f_0$
		\item For all $t\geq 0$, mass is conserved, $\int_{\R^d} f(t,v) \, \mathrm{d} v = \int_{\R^d} f_0(v) \, \mathrm{d}v$, kinetic energy is conserved, $\int_{\R^d} f(t,v)\, v^2 \, \mathrm{d} v = \int_{\R^d} f_0(v) \, v^2 \, \mathrm{d}v$, and the entropy is increasing, that is, $\mathcal{H}(f)$ is decreasing, $\mathcal{H}(f(t,\cdot))\leq \mathcal{H}(f_0)$.
		\item For all $\varphi\in\mathcal{C}^1(\R_+; \mathcal{C}_0^{\infty}(\R^d))$ one has
		\begin{align}\label{eq:weakformulation}
		\begin{split}
			&\langle f(t,\cdot), \varphi(t,v) \rangle - \langle f_0 , \varphi(0,\cdot)\rangle - \int_0^t \langle f(\tau, \cdot) \partial_{\tau}\varphi(\tau,\cdot)\rangle \,  \mathrm{d}\tau \\
			&\quad = \int_0^t \langle Q(f,f)(\tau, \cdot), \varphi(\tau,\cdot) \rangle \, \mathrm{d}\tau, \quad \text{for all } t\geq 0,
			\end{split}
		\end{align}
		where the latter expression involving $Q$ is defined for test functions $\varphi \in W^{2,\infty}(\R^d)$ by
		\begin{align*}
			&\langle Q(f,f), \varphi \rangle \\
			&\quad = \frac{1}{2}\int_{\R^{2d}}\int_{\S^{d-1}} b\left(\frac{v-v_*}{|v-v_*|}\cdot \sigma\right) f(v_*) f(v) \left( \varphi(v')+\varphi(v'_*) - \varphi(v) - \varphi(v_*)\right) \, \mathrm{d}\sigma \mathrm{d}v\mathrm{d}v_*.
		\end{align*}
	\end{enumerate}
\end{definition}
Weak solutions of the above type of the Cauchy problem \eqref{eq:cauchyproblem} for the  homogeneous Boltzmann equation are known to exist due to results by \textsc{Arkeryd} \cite{Ark72,Ark81}, which were later extended by \textsc{Villani} \cite{Vil98}. They are known to be unique \cite{TV99}, see  also the review articles \cite{Des01,MW99}.

In \cite{MUXY09} it has been shown that weak solutions to the Cauchy problem \eqref{eq:cauchyproblem} with Debye-Yukawa type interactions enjoy an $H^{\infty}$ smoothing property, i.e. starting with arbitrary initial datum $f_0\geq 0$, $f_0\in L^1_2\cap L\log L$, one has $f(t,\cdot) \in H^{\infty}$ for any positive time $t>0$. 

Based upon our recent proof \cite{BHRV15} of Gevrey smoothing for the homogeneous Boltzmann equation with Maxwellian molecules and angular singularity of the inverse-power law type \eqref{eq:powerlaw}, we can show a stronger than $H^{\infty}$ regularisation property of weak solutions in the Debye-Yukawa case. 

To this aim we define the function spaces
\begin{definition}
	Let $\mu>0$. A function $f\in H^{\infty}(\R^d)$ is defined to be in the space $\mathcal{A}^{\mu}(\R^d)$ if there exist constants $C>0$ and $b>0$ such that 
	\begin{align}\label{eq:derivatives}
		\left\| \partial^{\alpha} f \right\|_{L^2} \leq C^{|\alpha|+1} \E^{b |\alpha|^{1+1/\mu}} \quad \text{for all } \alpha\in\N_0^d.
	\end{align}
\end{definition}

For $\mu>0$ we define the family of function spaces, parametrised by $\tau>0$,
\begin{align*}
	\mathcal{D}\left( \E^{\tau \left(\log\langle D\rangle\right)^{\mu+1}}: L^2(\R^d) \right) \coloneq \left\{ f\in L^2(\R^d): \E^{\tau \left(\log\langle D\rangle\right)^{\mu+1}}f \in L^2(\R^d) \right\}
\end{align*}
\begin{remark}\label{rem:functionspace}
	Let $\mu>0$. Then 
	\begin{align*}
		\mathcal{A}^{\mu}(\R^d) = \bigcup_{\tau>0} \mathcal{D}\left( \E^{\tau \left(\log\langle D\rangle\right)^{\mu+1}}: L^2(\R^d) \right).
	\end{align*}
	The proof is rather technical and is deferred to Appendix \ref{app:analyticity}.
\end{remark}

In view of the coercivity property \eqref{eq:intuition2} and the regularisation properties of the logarithmic heat equation 
\begin{align}\label{eq:heat}
	\partial_t f = (\log(1-\Delta))^{\mu+1} f,
\end{align}
	the spaces $\mathcal{A}^{\mu}$, through their Fourier characterisation in Remark \ref{rem:functionspace}, capture exactly the gain of regularity that is to be expected for the Boltzmann equation with Debye-Yukawa type angular singularity. 
Indeed, our main result is
	
\begin{theorem}\label{thm:main} Let $f$ be a weak solution of the Cauchy problem \eqref{eq:cauchyproblem} for the homogeneous Bolzmann equation for Maxwellian molecules with angular collision kernel satisfying \eqref{eq:singularity} and \eqref{eq:momentumtransfer}, and initial datum $f_0\geq 0$, $f_0\in L^1_2(\R^d) \cap L\log L(\R^d)$. Then for any $T_0>0$ there exist $\beta,M>0$ such that 
	\begin{align*}
		\E^{\beta t \left(\log\langle D_v\rangle \right)^{\mu+1}} f(t,\cdot) \in L^2(\R^d)
		\intertext{and}
		\sup_{\eta\in\R^d} \E^{\beta t \left(\log\langle D_v\rangle \right)^{\mu+1}} |\hat{f}(t,\eta)| \leq M
	\end{align*} 
	for all $t\in(0,T_0]$. In particular, $f(t,\cdot) \in\mathcal{A}^{\mu}$ for all $t>0$.
\end{theorem}
\begin{remark}
  This regularity is much weaker than the Gevrey regularity we proved in \cite{BHRV15} for singular kernels of the form \eqref{eq:powerlaw}, but it is much stronger than the $H^\infty$ smoothing shown in \cite{MUXY09}. Moreover, it is \emph{exactly} the right type of regularity one would expect for a coercive term of the form \eqref{eq:intuition2} from the analogy with the heat equation \eqref{eq:heat}.
\end{remark}

For our proof we have to choose $\beta$ small if $T_0$ is large and our bounds on $\beta$ deteriorate to zero in the limit $T_0\to\infty$, so our Theorem \ref{thm:main} does not give a uniform result for all $t>0$. Nevertheless,  by propagation results due to \textsc{Desvillettes, Furiolo} and \textsc{Terraneo} \cite{DFT09} we even have the uniform bound
\begin{corollary}\label{cor:main}
	Under the same assumptions as in Theorem \ref{thm:main}, for any weak solution $f$ of the Cauchy problem \eqref{eq:cauchyproblem} with 
	initial datum $f_0\ge 0$ and $f_0\in L^1_2(\R^d)\cap L\log L(\R^d)$, there exist constants $0<K,C<\infty$ such that
	\begin{align}\label{eq:Fourier-uniform}
		\sup_{0\le t<\infty}\sup_{\eta\in\R^d}
		\E^{K \min(t,1) \, \left(\log\langle \eta \rangle\right)^{\mu+1}} \,|\hat{f}(t,\eta)|
		\leq C .
	\end{align}
\end{corollary}

The strategy of the proofs of our main result Theorem \ref{thm:main} is as follows:
We start with the additional assumption $f_0\in L^2$ on the initial datum (Theorem \ref{thm:mainL2}). We use the
known $H^\infty$ smoothing \cite{MUXY09} of the non-cutoff Boltzmann equation to allow for this.
Within an $L^2$ framework, a reformulation of the weak formulation of the Boltzmann equation is possible which includes suitable growing Fourier multipliers. As in \cite{MUXY09} the inclusion of Fourier multipliers leads to a nonlocal and nonlinear commutator with the Boltzmann kernel.
For non-power-type Fourier multipliers this commutator is considerably more complicated than the one encountered in the $H^{\infty}$ smoothing case. To overcome this, we follow the strategy we developed in \cite{BHRV15}, where an inductive procedure was invented to control the commutation error, in order to prove the Gevrey smoothing conjecture in the Maxwellian molecules case. 

The main differences compared with \cite{BHRV15} are:
\begin{enumerate}[label=(\arabic*)]
	\item For the weights needed  in the proof of Theorem \ref{thm:main} we have a much stronger enhanced subadditivity bound, see Lemma \ref{lem:subadditivity}. The proof is more involved than the one in \cite{BHRV15}, though.
	\item Because of the stronger form of the subadditivity bound, we can allow for a bigger loss in the induction step. We can therefore work with a more straightforward version of the `impossible' $L^2$-to-$L^{\infty}$ bound, see Lemma \ref{lem:impossible-embedding}.
	\item Due to the special form of the weights we use in this paper, which are in some sense in between the power type weights used in \cite{MUXY09} and the sub-gaussian weight used in \cite{BHRV15}, we don't have to do much of the additional songs and dances from \cite{BHRV15}. 
\end{enumerate}

\section{Enhanced subadditivity and properties of the Fourier weights}

\begin{lemma}\label{lem:subadditivity}
	Let $\mu>0$ and $h: [0, \infty) \to [0, \infty)$, $s \mapsto h(s) = \left(\log(\alpha+s)\right)^{\mu+1}$ for some $\alpha \geq e^{\mu}$. Then $h$ is increasing, concave and for any $0\leq s_- \leq s_+$,
	\begin{align}\label{eq:subadditivity-weight}
		h(s_- + s_+) \leq \frac{\mu+1}{1+\log \alpha} h(s_-) + h(s_+).
	\end{align}
\end{lemma}

\begin{remark}
For $\alpha \geq \E^{\mu}$, one has $h(0) = \mu^{\mu+1} >0$, and from the concavity of $h$ one concludes the subadditivity estimate 
\begin{align*}
	h(s_-) + h(s_+) \geq h(s_- + s_+) + h(0) > h(s_- + s_+)
\end{align*}
for all $s_-, s_+ \geq 0$. Note that this is the best possible bound for general $s_-, s_+ \geq 0$.
For $0\leq s_- \leq s_+$ Lemma \ref{lem:subadditivity} shows that the subadditivity bound can be improved to gain the small factor $\frac{\mu+1}{1+\log\alpha}$, which is strictly less than one for $\alpha > \E^{\mu}$, in front of $h(s_-)$. So this is indeed an \emph{enhanced} subadditivity property of the function $h$.  

Lemma \ref{lem:subadditivity} plays a similar role in the proof of Theorem \ref{thm:main}, as Lemma 2.6 in our previous paper \cite{BHRV15}. Here the situation is a bit simpler than in \cite{BHRV15}, since by choosing $\alpha$ large enough, we can make the term $\frac{\mu+1}{1+\log \alpha}$ as small as we like.
\end{remark}

\begin{proof}
	Since
	\begin{align*}
		h'(s) &= \frac{\mu+1}{\alpha+s} \left(\log(\alpha+s)\right)^{\mu} \geq 0 \quad \text{if } \alpha\geq 1,
	\end{align*}
	the function $h$ is increasing. Further, 
	\begin{align*}
		h''(s) &= \frac{\mu+1}{(\alpha+s)^2} \left(\log(\alpha+s)\right)^{\mu-1} \left(\mu - \log(\alpha+s)\right) \leq \frac{\mu+1}{(\alpha+s)^2} \left(\log(\alpha+s)\right)^{\mu-1} \left(\mu - \log(\alpha)\right) \leq 0
	\end{align*}
	for $\alpha\geq \E^{\mu}$, so $h$ is concave. 
	
	For all $s_-, s_+ \geq 0$,
	\begin{align*}
		h(s_- + s_+) = h(s_-) \frac{h(s_- + s_+) - h(s_+)}{h(s_-)} + h(s_+),
	\end{align*}
	and by concavity, $s_+ \mapsto h(s_- + s_+) - h(s_+)$ is decreasing, so using $0 \leq s_- \leq s_+$ one has 
	\begin{align*}
		h(s_- + s_+) \leq h(s_-) \frac{h(2s_-) - h(s_-)}{h(s_-)} + h(s_+).
	\end{align*}
	Since $h'$ is decreasing,
	\begin{align*}
		h(2s_-) - h(s_-) = \int_{s_-}^{2s_-} h'(r)\,\D{r} \leq h'(s_-)s_-
	\end{align*}
	and we get
	\begin{align*}
		h(s_- + s_+) \leq h(s_-) \frac{h'(s_-) s_-}{h(s_-)} + h(s_+) = h(s_-) \frac{(\mu+1) s_-}{(\alpha+s_-) \log(\alpha+s_-)} + h(s_+).
	\end{align*}
	For $\alpha\geq 1$ the function $F_{\alpha}: [0, \infty) \to \R$, $F_{\alpha}(s)\coloneq  (\alpha+s) \log(\alpha+s)$, is strictly convex and thus 
	\begin{align*}
		F_{\alpha}(s) \geq F_{\alpha}(0) + F_{\alpha}'(0) s = \alpha \log\alpha + (1+\log\alpha)s \geq (1+\log\alpha) s.
	\end{align*}
	It follows that $\frac{s_-}{(\alpha+s_-)\log(\alpha+s_-)} \leq \frac{1}{1+\log\alpha}$ and therefore
	\begin{align*}
		h(s_- + s_+) \leq h(s_-) \frac{\mu+1}{1+\log\alpha} + h(s_+).
	\end{align*}
\end{proof}

\begin{proposition} \label{prop:ce-bound}
Let $\beta, t, \mu > 0$, $\alpha \geq \E^{\mu}$ and define the function $\widetilde{G}: [0, \infty) \to \R$ by 
	\begin{align*}
		\widetilde{G}(r)\coloneq  \E^{\beta t 2^{-\mu-1} \left(\log(\alpha+r)\right)^{\mu+1}}.
	\end{align*}
Then for all $0\leq s_- \leq s_+$ with $s_-+s_+ = s$ one has
\begin{align*}
	\left| \widetilde{G}(s) - \widetilde{G}(s_+)\right| \leq 2^{-\mu} \beta t (\mu+1) \left(1-\frac{s_+}{s} \right) \left( \log(\alpha+s) \right)^{\mu} \widetilde{G}(s_-)^{\frac{\mu+1}{1+\log\alpha}} \widetilde{G}(s_+).
\end{align*}
\end{proposition}

\begin{proof}
	Using
	\begin{align*}
		\widetilde{G}'(s) = 2^{-\mu-1} \beta t (\mu+1) \frac{1}{\alpha+s} \left( \log(\alpha+s) \right)^{\mu} \widetilde{G}(s)
	\end{align*}
	one has
	\begin{align*}
		\widetilde{G}(s) - \widetilde{G}(s_+) = \int_{s_+}^s \widetilde{G}'(r)\,\D{r} \leq 2^{-\mu-1} \beta t (\mu+1) \frac{s-s_+}{\alpha+s_+} \left( \log(\alpha+s) \right)^{\mu} \widetilde{G}(s),
	\end{align*}
	where we used that $s_+\leq s$ and the fact that $\widetilde{G}$ is increasing. Since $s_-+s_+ = s$ and $0\leq s_- \leq s_+$, in particular $s_+\geq \frac{s}{2}$, we can further estimate
	\begin{align*}
		\frac{s-s_+}{\alpha+s_+} = \left(1- \frac{s_+}{s} \right) \frac{s}{\alpha+s_+} \leq \left(1- \frac{s_+}{s} \right) \frac{2s_+}{\alpha+s_+} \leq 2\left(1- \frac{s_+}{s} \right),
	\end{align*}
	to obtain 
	\begin{align*}
		\widetilde{G}(s) - \widetilde{G}(s_+) \leq 2^{-\mu} \beta t (\mu+1) \left( 1 - \frac{s_+}{s} \right) \left( \log(\alpha+s) \right)^{\mu} \widetilde{G}(s).
	\end{align*}
	The rest now follows from the enhanced subadditivity property \eqref{eq:subadditivity-weight}, namely
	\begin{align*}
		\widetilde{G}(s) = \widetilde{G}(s_- + s_+) \leq \widetilde{G}(s_-)^{\frac{\mu+1}{1+\log\alpha}} \widetilde{G}(s_+). 
	\end{align*}
\end{proof}

\section{Extracting $L^{\infty}$ bounds from $L^2$: a simple proof}
Following is a simple bound which controls the size of  a function $h$ in terms of its local $L^2$ norm and some global a priori bounds on $h$ and its derivative. 
\begin{lemma} \label{lem:impossible-embedding}
	Let $h\in \mathcal{C}^1_b(\R^d)$, i.e. $h$ is a bounded continuously differentiable function with bounded derivative. Then there exists a constant $L<\infty$ (depending only on $d, \|h\|_{L^{\infty}(\R^d)}$ and, $\|\nabla h\|_{L^{\infty}(\R^d)}$\,) such that for any $x\in\R^d$,
	\begin{align}\label{eq:impossible-embedding}
		|h(x)| \leq L \left( \int_{Q_x} |h(y)|^2 \,\mathrm{d}y \right)^{\frac{1}{d+2}},
	\end{align}
	where $Q_x$ is a unit cube in $\R^d$ with $x$ being one of the corners, oriented away from the origin in the sense that $x\cdot (y-x) \geq 0$ for all $y\in Q_x$. 
\end{lemma}

\begin{remarks}
\begin{SE}
	\item We use the norm $\|\nabla h\|_{L^{\infty}(\R^d)} = \sup_{\eta\in\R^d} |\nabla h(\eta)|$, where $|\cdot|$ is the Euclidean norm on $\R^d$.
	\item The exponent $\frac{1}{d+2}$ can be improved if higher derivatives of the function $h$ are bounded, see Section 2.3 in \cite{BHRV15}. This was important for the results of \cite{BHRV15}, but we don't need it here because of the stronger form of the enhanced subadditivity Lemma for the weight we consider in this paper. 
\end{SE}
\end{remarks}

\begin{remark} \label{rem:impossible-embedding}
	If $f\in L^1_1(\R^d)$, its Fourier transform satisfies $\hat{f}\in\mathcal{C}^1_b(\R^d)$ by the Riemann-Lebesgue lemma. Since $\nabla_\eta \hat{f}(\eta) =  \widehat{2\pi\I vf}(\eta)$ one has the a priori bound $\|\nabla \hat{f}\|_{L^\infty(\R^d)} \leq 2\pi \|f\|_{L^1_1(\R^d)}$. 
	
	If $f$ is a weak solution of the homogeneous Boltzmann equation, we can also bound $\|\nabla \hat{f}\|_{L^{\infty}(\R^d)} \leq  2\pi \|f_0\|_{L^1_2(\R^d)}$ uniformly in time due to conservation of energy. 
	
\end{remark}

\begin{proof}
	We first consider the one-dimensional case and prove the $d$-dimensional result by iteration in each coordinate direction. 
	
	Let $u\in\mathcal{C}^1_b(\R)$ and $q\geq 1$. Then for any $r\in\R$ we have 
	\begin{align}\label{eq:impossible-1d}
		|u(r)|^q \leq \max\left\{ q \|u'\|_{L^{\infty}(\R)}, \|u\|_{L^{\infty}(\R)}\right\} \int_{I_r} |u(s)|^{q-1}\,\D{s},
	\end{align} 
	where $I_r = [r, r+1]$ if $r\geq0$ and $I_r = [r-1,r]$ if $r<0$. 
	
	Indeed, assuming for the moment $r\geq0$, 
	\begin{align*}
		|u(r)|^q - \int_{I_r} |u(s)|^q\,\D{s} \leq \int_{I_r} |u^q(r) - u^q(s)|\,\D{s},
	\end{align*}
	and by the fundamental theorem of calculus,
	\begin{align*}
		|u^q(r) - u^q(s)| \leq q \int_{I_r} |u(t)|^{q-1} |u'(t)|\,\D{t} \leq q \|u'\|_{L^{\infty}(\R)} \int_{I_r} |u(t)|^{q-1}\,\D{t}.
	\end{align*}
	Combined with the trivial estimate $\int_{I_r} |u(s)|^{q}\,\D{s} \leq \|u\|_{L^{\infty}(\R)} \int_{I_r} |u(s)|^{q-1}\,\D{s}$ one arrives at inequality \eqref{eq:impossible-1d} for $r\geq 0$. The case $r<0$ is analogous. 
	
	For the case $d>1$ we remark that for any $y\in\R^d$, 
	\begin{align*}
		\|h(y_1, \dots, y_{j-1}, \,\cdot\,, y_{j+1}, \dots, y_d)\|_{L^{\infty}(\R)} &\leq \|h\|_{L^{\infty}(\R^d)}, \quad \text{and} \\
		\|\partial_j h(y_1, \dots, y_{j-1}, \,\cdot\,, y_{j+1}, \dots, y_d)\|_{L^{\infty}(\R)} &\leq \|\nabla h\|_{L^{\infty}(\R^d)}
	\end{align*}
	and setting $q=d+2$ iterative application of \eqref{eq:impossible-1d} in each coordinate direction yields for $x\in\R^d$
	\begin{align*}
		|h(x)|^{d+2} \leq \left(\max\left\{ (d+2) \|\nabla h\|_{L^{\infty}(\R^d)}, \|h\|_{L^{\infty}(\R^d)} \right\} \right)^d \int_{I_{x_1}\times \cdots \times I_{x_d}} |h(y)|^{d+2-d}\,\D{y},
	\end{align*}
	hence
	\begin{align*}
		|h(x)| \leq \left(\max\left\{ (d+2) \|\nabla h\|_{L^{\infty}(\R^d)}, \|h\|_{L^{\infty}(\R^d)} \right\} \right)^{\frac{d}{d+2}} \left(\int_{Q_x} |h(y)|^2 \,\D{y}\right)^{\frac{1}{d+2}} =: L \|h\|_{L^2(Q_x)}^{\frac{2}{d+2}}
	\end{align*}
	where $Q_x = I_{x_1} \times \cdots \times I_{x_d}$ is a unit cube directed away from the origin with $x\in\R^d$ at one of its corners. 
\end{proof}

\section{Smoothing property of the Boltzmann operator}
A central step in the proof of Theorem \ref{thm:main} is to prove a version for $L^2$ initial data first. This is the content of Theorem \ref{thm:mainL2} below. In the remainder of this article we will always assume that the collision kernel satisfies assumptions \eqref{eq:singularity} and \eqref{eq:momentumtransfer}.

\begin{theorem}\label{thm:mainL2}
	Let $f$ be a weak solution of the Cauchy problem \eqref{eq:cauchyproblem} with initial datum $f_0\geq 0$, $f_0\in L^1_2(\R^d) \cap L\log L(\R^d)$ and in addition $f_0\in L^2(\R^d)$. 
	
	Then for all $T_0>0$ there exist $\beta,M>0$ such that for all $t\in[0,T_0]$
	\begin{align*}
		\sup_{\eta\in\R^d} \E^{\beta t \left(\log\langle \eta \rangle_{\alpha}\right)^{\mu+1}} |\hat{f}(t,\eta)| \leq M
	\end{align*}
	and 
	\begin{align*}
		\E^{\beta t \left(\log\langle D_v \rangle_{\alpha}\right)^{\mu+1}} f(t,\cdot) \in L^2(\R^d)
	\end{align*}
	where $\alpha = \E^{\frac{d}{2} + \frac{d+2}{2}\mu}$.
\end{theorem}
We give the proof of Theorem \ref{thm:mainL2} in section \ref{sec:mainL2}.  
To prepare for its proof, let $\alpha \geq \E^{\mu}$ and $\beta>0$ and define the Fourier multiplier $G: \R_+ \times \R^d \to \R_+$ by 
\begin{align*}
	G(t,\eta) \coloneq  \E^{\beta t \left(\log\langle\eta\rangle_{\alpha}\right)^{\mu+1}}, \quad \langle\eta\rangle_{\alpha}\coloneq \left(\alpha+|\eta|^2\right)^{\frac{1}{2}}
\end{align*}
and for $\Lambda > 0$ the cut-off multiplier $G_{\Lambda}: \R_+ \times \R^d \to [0, \infty)$ by
\begin{align*}
	G_{\Lambda}(t,\eta) \coloneq  G(t,\eta) \1_{\Lambda}(|\eta|)
\end{align*}
wehre $\1_{\Lambda}$ is the characteristic function of the interval $[0, \Lambda]$. The associated Fourier multiplication operator is denoted by $G_{\Lambda}(t,D_v)$,
\begin{align*}
	G_{\Lambda}(t, D_v)f \coloneq  \mathcal{F}^{-1}\left[ G_{\Lambda}(t, \cdot) \hat{f}(t,\cdot) \right]
\end{align*}

By Bobylev's identity, the Fourier transform of the Boltzmann operator for Maxwellian molecules is
\begin{align}\label{eq:bobylev}
	\widehat{Q(g,f)}(\eta) = \int_{\S^{d-1}} b\left(\frac{\eta}{|\eta|}\cdot\sigma\right) \left[ \hat{g}(\eta^-) \hat{f}(\eta^+) - \hat{g}(0) \hat{f}(\eta) \right] \, \mathrm{d}\sigma, \quad \eta^{\pm} = \frac{\eta\pm |\eta| \sigma}{2},
\end{align}


Note that, due to the cut-off in Fourier space,
\[G_{\Lambda}f, G_{\Lambda}^2 f \in L^{\infty}([0,T_0]; H^{\infty}(\R^d))\]
for any finite $T_0>0$ and $\Lambda>0$, if $f\in L^{\infty}([0,T_0];L^1(\R^d))$, and even analytic in a strip containing $\R^d_v$. In particular, by Sobolev embedding, $G_{\Lambda}f, G_{\Lambda}^2f \in L^{\infty}([0, T_0]; W^{2,\infty}(\R^d))$, so 
\begin{align*}
	\left\langle Q(f,f)(t,\cdot) , G_{\Lambda}^2 f(t, \cdot) \right\rangle 
\end{align*}
is well-defined.

\subsection{$L^2$ reformulation and coercivity}

\begin{proposition}\label{prop:L2reform}
	Let $f$ be a weak solution of the Cauchy problem \eqref{eq:cauchyproblem} with initial datum $f_0$ satisfying $0\leq f_0\in L^1_2(\R^d) \cap L\log L(\R^d)$, and let $T_0>0$. Then for all $t\in(0, T_0]$, $\beta>0$, $\alpha\in(0,1)$, and $\Lambda>0$ we have
$G_{\Lambda}f \in \mathcal{C}\left([0, T_0]; L^2(\R^d)\right)$ and
	\begin{align}\label{eq:reformulation}
	\begin{split}
	&\frac{1}{2} \|G_{\Lambda}(t,D_v)f(t,\cdot)\|_{L^2}^2 - \frac{1}{2} \int_0^t \left\langle f(\tau, \cdot), \left( \partial_\tau G_{\Lambda}^2(\tau, D_v) \right) f(\tau,\cdot)\right\rangle \,\mathrm{d}\tau \\
	&= \frac{1}{2} \|\1_{\Lambda}(D_v)f_0\|_{L^2}^2 + \int_0^t \left\langle Q(f,f)(\tau, \cdot), G_{\Lambda}^2(\tau, D_v)f(\tau, \cdot)\right\rangle \, \mathrm{d}\tau.
	\end{split}
	\end{align}
\end{proposition}
Informally, equation \eqref{eq:reformulation} follows from using $\varphi(t,\cdot)\coloneq  G_{\Lambda}^2(t,D_v)f(t,\cdot)$  in the weak formulation of the homogenous Boltzmann equation.
Recall that $G_{\Lambda}^2f \in L^{\infty}([0, T_0]; W^{2,\infty}(\R^d))$ for any finite $T_0>0$, so it still misses the required regularity in time needed to be used as a test function. The proof of Proposition \ref{prop:L2reform} is analogous to \textsc{Morimoto} \textit{et al.} \cite{MUXY09}, see also Appendix A in \cite{BHRV15}.

For weak solutions of the homogeneous Boltzmann equation we have (see also Corollary \ref{cor:subelliptic}):

\begin{proposition}\label{prop:coercivity}
	Let $g$ be a weak solution of the Cauchy problem \eqref{eq:cauchyproblem} with initial datum $g_0\in L^1_2(\R^d)\cap L\log L(\R^d)$. Then there exist constants $C_{g_0},\widetilde{C}_{g_0}>0$ depending only on the dimension $d$, the angular collision kernel $b$, $\|g_0\|_{L^1}$, $\|g_0\|_{L^1_2}$ and $\|g_0\|_{L\log L}$ such that for all $f\in H^1(\R^d)$ one has
	\begin{align}\label{eq:coercivity}
	-\langle Q(g,f),f \rangle \geq \frac{C_{g_0}}{\left(\log( \alpha + \E )\right)^{\mu+1}} \left\|\left(\log\langle D_v\rangle_{\alpha}\right)^{\frac{\mu+1}{2}} f \right\|_{L^2}^2 - \widetilde{C}_{g_0}  \|f\|_{L^2}^2,
\end{align}
uniformly in $t\geq 0$.
\end{proposition}

\begin{remark}
The above estimate makes the intuition \eqref{eq:intuition2} on the coercivity of the Boltzmann collision operator precise. It was already used in \textsc{Motimoto, Ukai, Xu} and \textsc{Yang} \cite{MUXY09} to show $H^{\infty}$ smoothing and goes back to \textsc{Alexandre, Desvillettes, Villani} and \textsc{Wennberg} \cite{ADVW00}, where they proved the corresponding sub-elliptic estimate for Boltzmann collision operators with the singularity arising from power-law interaction potentials and more general singularities. 

Since we need to carefully fine-tune some of the constants in our inductive procedure, we need a precise information about the dependence of the constants on $\alpha$ in this inequality. Therefore we will give the proof of the coercivity estimate in the form stated above in Appendix \ref{app:coercivity}. 	
\end{remark}

Together with Proposition \ref{prop:L2reform} the coercivity estimate from Proposition \ref{prop:coercivity} implies

\begin{corollary}[A priori bound for weak solutions]  \label{cor:gronwallbound}
	Let $f$ be a weak solution of the Cauchy problem \eqref{eq:cauchyproblem} with initial datum $f_0\geq 0$ satisfying $f_0 \in L^1_2\cap L\log L$, and let $T_0>0$. Then there exist constants $\widetilde{C}_{f_0}, C_{f_0} >0$ (depending only on the dimension $d$, the collision kernel $b$, $\|f_0\|_{L^1_2}$ and $\|f_0\|_{L\log L}$) such that for all $t\in(0, T_0]$, $\beta,\mu>0$, $\alpha\geq 0$, and $\Lambda>0$ we have
	\begin{align}\label{eq:aprioribound}
	\begin{split}
		\|G_{\Lambda}f\|_{L^2}^2
		\leq
		  \|\1_{\Lambda}(D_v)f_0\|_{L^2}^2
		  &+ 2 \widetilde{C}_{f_0} \int_0^t \|G_{\Lambda}f\|_{L^2}^2 \, \mathrm{d}\tau \\
		&+ 2 \int_0^t \left( \beta - \frac{C_{f_0}}{(\log(\E+\alpha))^{\mu+1}} \right) \left\| \left(\log\langle D_v \rangle_{\alpha} \right)^{\frac{\mu+1}{2}} G_{\Lambda} f \right\|_{L^2}^2 \, \mathrm{d}\tau \\
		&+ 2 \int_0^t \left\langle G_{\Lambda} Q(f,f) -  Q(f,G_{\Lambda}f), G_{\Lambda}f\right\rangle \, \mathrm{d}\tau.
		\end{split}
	\end{align}
\end{corollary}

\begin{proof}
In order to make use of the coercivity property of the Boltzmann collision operator, we write 
\begin{align*}
	\langle Q(f,f), G_{\Lambda}^2f \rangle &= \langle G_{\Lambda} Q(f,f), G_{\Lambda}f\rangle \\ &= \langle Q(f, G_{\Lambda}f), G_{\Lambda}f\rangle + \langle G_{\Lambda}Q(f,f) - Q(f,G_{\Lambda}f), G_{\Lambda}f\rangle
\end{align*}
and estimate the first term with \eqref{eq:coercivity}. 

Since $\partial_\tau G_{\Lambda}^2(\tau, \eta)= 2\beta \left( \log\langle\eta\rangle_{\alpha} \right)^{\mu+1} G_\Lambda^2(t,\eta)$,
we further have 
\begin{align*}
\left\langle f, \left(\partial_\tau G_{\Lambda}^2 \right) f\right\rangle = 2 \beta \left\| \left(\log\langle D_v \rangle_{\alpha} \right)^{\frac{\mu+1}{2}} G_{\Lambda}f \right\|_{L^2}^2,
\end{align*}
and inserting those two results into \eqref{eq:reformulation}, one obtains the claimed inequality \eqref{eq:aprioribound}.
\end{proof}

\subsection{Controlling the commutation error}\label{ssec:ce}

\begin{proposition}[Bound on the Commutation Error]
	Let $f$ be a weak solution of the Cauchy problem \eqref{eq:cauchyproblem} with initial datum $f_0 \geq 0$, $f_0\in L^1_2(\R^d) \cap L \log L(\R^d)$. Then for all $t, \beta, \mu,\Lambda>0$ and $\alpha\geq \E^{\mu}$ one has the bound
	\begin{align}\label{eq:ce-estimate}
	\begin{split}
		|\langle &G_{\Lambda}Q(f,f) - Q(f,G_{\Lambda}f), G_{\Lambda}f\rangle| \\
		&\leq \beta t (\mu+1) \int_{\R^d} \int_{\S^{d-1}} b\left(\frac{\eta}{|\eta|}\cdot\sigma\right) \, \left(1-\frac{|\eta^+|^2}{|\eta|^2} \right) \left(\log\langle\eta\rangle_{\alpha}\right)^{\mu} \, G(\eta^-)^{\frac{\mu+1}{1+\log\alpha}} |\hat{f}(\eta^-)| \\
		&\hskip20em \times  G_{\Lambda}(\eta^+) |\hat{f}(\eta^+)| \, G_{\Lambda}(\eta) |\hat{f}(\eta)| \,\D\sigma\D\eta.
	\end{split}
	\end{align}
\end{proposition}

\begin{remark}
	The bound \eqref{eq:ce-estimate} is very similar to the one we derived in \cite{BHRV15}. In particular, it is a trilinear expression in the weak solution $f$. The $\hat{f}(\eta^-)$ term is multiplied by a \emph{faster-than-polynomially} growing function. If the Fourier multiplier $G$ were only growing \emph{polynomially}, the factor $G(\eta^-)^{\frac{\mu+1}{1+\log\alpha}}$ would be replaced by $1$, making the analysis \emph{much easier}. We will therefore rely on the inductive procedure we developed in \cite{BHRV15} to treat exactly this type of situation. 
\end{remark}

\begin{proof}
	Bobylev's identity and a small computation show that
	\begin{align*}
		&\left| \left\langle G_{\Lambda}Q(f, f) - Q(f,G_{\Lambda} f), G_{\Lambda}f\right\rangle\right| = \left| \left\langle \mathcal{F}\left[G_{\Lambda}Q(f, f) - Q(f,G_{\Lambda} f)\right], \mathcal{F}\left[G_{\Lambda}f\right]\right\rangle_{L^2}\right| \\
		&\quad \leq \int_{\R^d} \int_{\S^{d-1}} b\left(\frac{\eta}{|\eta|}\cdot\sigma\right) G_{\Lambda}(\eta) |\hat{f}(\eta)| \, |\hat{f}(\eta^-)| \, |\hat{f}(\eta^+)| |G(\eta) - G(\eta^+)| \,\mathrm{d}\sigma \,\mathrm{d}\eta
	\end{align*}
since $G_{\Lambda}$ is supported on the ball $\{|\eta|\leq \Lambda\}$ and $|\eta^+| \leq |\eta|$. We further have 
\begin{align*}
	|\eta^\pm|^2 = \frac{|\eta|^2}{2} \left(1\pm\frac{\eta\cdot\sigma}{|\eta|}\right), \quad |\eta^-|^2 + |\eta^+|^2 = |\eta|^2,
\end{align*} 
in particular by the support assumption on the collision kernel $b$, $\frac{\eta\cdot\sigma}{|\eta|} \in [0,1]$, and therefore 
\begin{align*}
	0\leq |\eta^-|^2 \leq \frac{|\eta|^2}{2} \leq |\eta^+|^2 \leq |\eta|^2. 
\end{align*}
From Proposition \ref{prop:ce-bound} it now follows that
\begin{align*}
	\left| G(\eta) - G(\eta^+) \right| &= \left| \widetilde{G}(|\eta|^2) - \widetilde{G}(|\eta^+|^2) \right| \\
	&\leq \beta t (\mu+1) \left(1-\frac{|\eta^+|^2}{|\eta|^2}\right) \left( \log\langle\eta\rangle_{\alpha} \right)^{\mu} \, G(\eta^-)^{\frac{\mu+1}{1+\log\alpha}} G(\eta^+),
\end{align*}
which completes the proof.
\end{proof}

\begin{lemma}\label{lem:ce}
	\begin{align}\label{eq:ce-estimate-2}
	\begin{split}
		\int_{\R^d} \int_{\S^{d-1}} &b\left(\frac{\eta}{|\eta|}\cdot\sigma\right) \, \left(1-\frac{|\eta^+|^2}{|\eta|^2} \right) \left(\log\langle\eta\rangle_{\alpha}\right)^{\mu} \,G_{\Lambda}(\eta^+) |\hat{f}(\eta^+)| \, G_{\Lambda}(\eta) |\hat{f}(\eta)| \,\D\sigma\D\eta \\
		&\leq c_{b,d} \left( \|G_{\Lambda}f\|_{L^2}^2 +  \left\| \left(\log\langle D_v \rangle_{\alpha} \right)^{\frac{\mu}{2}} G_{\Lambda}f\right\|_{L^2}^2 \right),
	\end{split}
	\end{align}
where $c_{b,d} = \frac{1}{2}\max\{1, 2^{\mu-1}\} \max\{2^{d-1-\mu} (\log 2)^{\mu}, 1+2^{d-1}\}\,  |\S^{d-2}|\int_{0}^{\frac{\pi}{2}}\sin^{d}\theta\, b(\cos\theta)\,\D\theta $.
\end{lemma}

\begin{proof}
	Using Cauchy-Schwartz, in the form $ab\leq \frac{a^2}{2} + \frac{b^2}{2}$, one can split the integral into 
	\begin{align*}
		&\int_{\R^d} \int_{\S^{d-1}} b\left(\frac{\eta}{|\eta|}\cdot\sigma\right) \, \left(1-\frac{|\eta^+|^2}{|\eta|^2} \right) \left(\log\langle\eta\rangle_{\alpha}\right)^{\mu} \,G_{\Lambda}(\eta^+) |\hat{f}(\eta^+)| \, G_{\Lambda}(\eta) |\hat{f}(\eta)| \,\D\sigma\D\eta \\
		&\leq \frac{1}{2} \int_{\R^d} \int_{\S^{d-1}} b\left(\frac{\eta}{|\eta|}\cdot\sigma\right) \, \left(1-\frac{|\eta^+|^2}{|\eta|^2} \right) \left(\log\langle\eta\rangle_{\alpha}\right)^{\mu} \, G_{\Lambda}(\eta)^2 |\hat{f}(\eta)|^2 \,\D\sigma\D\eta \\
		&\quad + \frac{1}{2} \int_{\R^d} \int_{\S^{d-1}} b\left(\frac{\eta}{|\eta|}\cdot\sigma\right) \, \left(1-\frac{|\eta^+|^2}{|\eta|^2} \right) \left(\log\langle\eta\rangle_{\alpha}\right)^{\mu} \,G_{\Lambda}(\eta^+)^2 |\hat{f}(\eta^+)|^2 \,\D\sigma\D\eta
	\end{align*}
	and we will treat the two terms separately. To estimate the first integral, one introduces polar coordinates such that $\frac{\eta}{|\eta|}\cdot\sigma = \cos\theta$ and thus, since 
	\begin{align*}
		|\eta^+|^2 = |\eta|^2 \left(1+\frac{\eta}{|\eta|}\cdot\sigma\right) = |\eta|^2 \cos^2\frac{\theta}{2},
	\end{align*}
	obtains
	\begin{align*}
		I&\coloneq \frac{1}{2} \int_{\R^d} \int_{\S^{d-1}} b\left(\frac{\eta}{|\eta|}\cdot\sigma\right) \, \left(1-\frac{|\eta^+|^2}{|\eta|^2} \right) \left(\log\langle\eta\rangle_{\alpha}\right)^{\mu} \, G_{\Lambda}(\eta)^2 |\hat{f}(\eta)|^2 \,\D\sigma\D\eta \\
		&= \frac{1}{2} |\S^{d-2}| \int_{0}^{\frac{\pi}{2}} \sin^{d-2}\theta \, b(\cos\theta) \, \sin^{2}\tfrac{\theta}{2} \,\D\theta \int_{\R^d}\left(\log\langle\eta\rangle_{\alpha}\right)^{\mu} \, G_{\Lambda}(\eta)^2 |\hat{f}(\eta)|^2 \,\D\eta \\
		&\leq  \frac{1}{2} |\S^{d-2}| \int_{0}^{\frac{\pi}{2}} \sin^{d}\theta \, b(\cos\theta) \, \D\theta \, \left\| \left(\log\langle D_v \rangle_{\alpha}\right)^{\frac{\mu}{2}} G_{\Lambda} f\right\|_{L^2}^2.
	\end{align*}
	Notice that the $\theta$ integral is finite due to the assumptions on the angular collision kernel. This is another instance where cancellation effects play an important role in controlling the singularity for grazing collisions. 
	
	It remains to bound the second integral, and we will do this after a change of variables $\eta \to \eta^+$. This change of variables is well-known to the experts, see, for example, \cite{ADVW00,MUXY09}. We give some details for the convenience of the reader. 
	
	Observe that $\frac{\eta^+\cdot\sigma}{|\eta^+|}= \frac{|\eta^+|}{|\eta|}$ and $\frac{\eta\cdot\sigma}{|\eta|} = 2\left(\frac{\eta^+\cdot\sigma}{|\eta^+|}\right)^2 - 1$, and by Sylvester's determinant theorem, one has
	\begin{align*}
		\left|\frac{\partial\eta^+}{\partial\eta}\right| = \left| \frac{1}{2} \left(1+\frac{\eta}{|\eta|} \otimes \sigma \right)\right| = \frac{1}{2^d} \left(1+\frac{\eta}{|\eta|} \cdot \sigma \right) = \frac{1}{2^{d-1}} \left(\frac{\eta^+\cdot\sigma}{|\eta^+|}\right)^2 = \frac{1}{2^{d-1}} \frac{|\eta^+|^2}{|\eta|^2}.
	\end{align*} 
	Since $0\leq |\eta^-| \leq |\eta^+|$ and $|\eta|^2 = |\eta^-|^2 + |\eta^+|^2$, in particular $|\eta|^2 \leq 2 |\eta^+|^2$, it follows that $\left|\frac{\partial\eta^+}{\partial\eta}\right| \geq 2^{-d}$ and
	\begin{align*}
		\log\langle\eta\rangle_{\alpha} = \tfrac{1}{2} \log(\alpha+|\eta|^2) \leq \tfrac{1}{2} \log{2} + \tfrac{1}{2} \log(\alpha + |\eta^+|^2) = \tfrac{1}{2} \log{2} + \log\langle \eta^+\rangle_{\alpha}.
	\end{align*}
	For all $x,y\geq 0$ one has 
	\begin{align*}
		\begin{cases}
			(x+y)^{\mu} \leq 2^{\mu-1} (x^{\mu}+y^{\mu}) & \text{for } \mu \geq 1 \text{ by convexity, and} \\
			(x+y)^{\mu} \leq x^{\mu} + y^{\mu} &\text{for } \mu<1,
		\end{cases}
	\end{align*}
	where the second statement is a consequence of the fact that for $0<\mu<1$ the function $0\leq s \mapsto h(s) = (1+s)^{\mu} - s^{\mu}$ is monotone decreasing for all $s>0$ with $h(0) =1$. Therefore,
	\begin{align*}
		\left(\log\langle \eta \rangle_{\alpha} \right)^{\mu} \leq \max\{1, 2^{\mu-1}\} \left(2^{-\mu}(\log 2)^{\mu} + \left(\log\langle\eta^+\rangle_{\alpha}\right)^{\mu}\right).
	\end{align*}
	After those preparatory remarks, we can estimate
	\begin{align*}
		I^+&\coloneq \frac{1}{2} \int_{\R^d} \int_{\S^{d-1}} b\left(\frac{\eta}{|\eta|}\cdot\sigma\right) \, \left(1-\frac{|\eta^+|^2}{|\eta|^2} \right) \left(\log\langle\eta\rangle_{\alpha}\right)^{\mu} \,G_{\Lambda}(\eta^+)^2 |\hat{f}(\eta^+)|^2 \,\D\sigma\D\eta \\
		&= \frac{1}{2} \int_{\S^{d-1}} \int_{\R^d} b\left(2\left(\frac{\eta^+\cdot\sigma}{|\eta^+|}\right)^2 - 1\right) \, \left(1-\left(\frac{\eta^+\cdot\sigma}{|\eta^+|}\right)^2 \right) \left(\log\langle\eta\rangle_{\alpha}\right)^{\mu} \,G_{\Lambda}(\eta^+)^2 |\hat{f}(\eta^+)|^2 \,\left|\frac{\partial \eta^+}{\partial\eta}\right|^{-1}\,\D\eta^+\D\sigma\\
		&\leq 2^{d-1} \max\{1, 2^{\mu-1}\} \\ 
		&\quad \times \Bigg[2^{-\mu}(\log{2})^{\mu}  \int_{\R^d} \int_{\S^{d-1}} b\left(2\left(\frac{\eta^+\cdot\sigma}{|\eta^+|}\right)^2 - 1\right) \, \left(1-\left(\frac{\eta^+\cdot\sigma}{|\eta^+|}\right)^2 \right) \,G_{\Lambda}(\eta^+)^2 |\hat{f}(\eta^+)|^2 \,\D\sigma\D\eta^+ \\
		&\qquad \quad  + \int_{\R^d} \int_{\S^{d-1}} b\left(2\left(\frac{\eta^+\cdot\sigma}{|\eta^+|}\right)^2 - 1\right) \, \left(1-\left(\frac{\eta^+\cdot\sigma}{|\eta^+|}\right)^2 \right) \left(\log\langle\eta^+\rangle_{\alpha}\right)^{\mu} \,G_{\Lambda}(\eta^+)^2 |\hat{f}(\eta^+)|^2 \,\D\sigma\D\eta^+ \Bigg].
	\end{align*}
	Introducing new polar coordinates with pole $\frac{\eta^+}{|\eta^+|}$, such that $\cos\vartheta = \frac{\eta^+\cdot\sigma}{|\eta^+|} \geq \frac{1}{\sqrt{2}}$, i.e. $\vartheta\in [0, \frac{\pi}{4}]$, one then gets
	\begin{align*}
		I^+ \leq 2^{d-1} &\max\{1, 2^{\mu-1}\} \, |\S^{d-2}| \int_{0}^{\frac{\pi}{4}}\sin^{d}\vartheta b(\cos 2\vartheta)\,\D\vartheta \\ 
		  &\quad \times \left[ 2^{-\mu}(\log{2})^{\mu} \| G_{\Lambda}f\|_{L^2}^2  + \left\| \left(\log\langle D_v\rangle_{\alpha}\right)^{\frac{\mu}{2}} G_{\Lambda}f\right\|_{L^2}^2 \right].
	\end{align*}
	Estimating $\int_{0}^{\frac{\pi}{4}}\sin^{d}\vartheta b(\cos 2\vartheta)\,\D\vartheta \leq \int_0^{\frac{\pi}{2}} \sin^{d}\theta\, b(\cos\theta)\,\D\theta$ and combining the bounds on $I$ and $I^+$ proves inequality \eqref{eq:ce-estimate-2}.
\end{proof}

\section{Smoothing effect for $L^2$ initial data: Proof of Theorem \ref{thm:mainL2}}
\label{sec:mainL2}
We now have all the necessary pieces together to start the inductive proof of Theorem \ref{thm:mainL2} for initial data that are in addition square integrable.

The proof is based on gradually removing the cut-off $\Lambda$ in Fourier space, in such a way that the commutation error can be controlled, even though it contains fast growing terms. 
For fixed $T_0,\mu>0$ and $\alpha \geq \E^{\mu}$ we define
\begin{definition}[Induction Hypothesis $\Hyp_{\Lambda}(M)$.]
	Let $M\geq 0$ and $\Lambda>0$. Then for all $0\leq t \leq T_0$,
	\begin{align*}
		\sup_{|\xi|\leq \Lambda} G(t,\xi)^{\frac{\mu+1}{1+\log\alpha}} |\hat{f}(t,\xi)| \leq M.
	\end{align*}
\end{definition}
\begin{remark}
Recall that the Fourier multiplier $G$ also depends on $\beta>0$ and $\alpha \geq \E^{\mu}$ and we suppress this dependence here.

The induction step itself will be divided into two separate steps: 
\begin{enumerate}[label=\textbf{Step \arabic*}]
	\item $\Hyp_{\Lambda}(M) \implies \|G_{\sqrt{2}\Lambda}f\|_{L^2} \leq C$ via a \emph{Gronwall argument}.
	\item \emph{$L^2 \to L^{\infty}$ bound:} $\|G_{\sqrt{2}\Lambda}f\|_{L^2}\leq C \implies \Hyp_{\widetilde{\Lambda}}(M)$ for intermediate $\widetilde{\Lambda} = \frac{1+\sqrt{2}}{2}\Lambda$.
\end{enumerate}
Here it is essential that $M$ does not increase during the induction procedure. This can be accomplished by choosing $\beta$ small enough at very beginning.

\end{remark}

\begin{lemma}[Step 1]\label{lem:step1}
	Fix $T_0,\mu>0$ and $\alpha \geq \E^{\mu}$ and let $M\geq 0$ and $\Lambda>0$. Let further $C_{f_0}, \widetilde{C}_{f_0}$ and $c_{b,d}$ be the constants from Corollary \ref{cor:gronwallbound} and Lemma \ref{lem:ce}, respectively. If 
	\begin{align}\label{eq:step1beta}
		0<\beta \leq \beta_0(\alpha) \coloneq  \frac{C_{f_0}}{(\log(\E+\alpha))^{\mu+1}} \frac{\log\alpha}{\log\alpha+ 2 T_0(\mu+1)c_{b,d}M},
	\end{align}
	then for any weak solution of the Cauchy problem \eqref{eq:cauchyproblem} with initial datum $f_0\geq 0$, $f_0\in L^1_2\cap L\log L$,
	\begin{align*}
		\Hyp_{\Lambda}(M) \implies \|G_{\sqrt{2}\Lambda} f\|_{L^2(\R^d)} \leq \|1_{\sqrt{2}\Lambda}(D_v)f_0\|_{L^2(\R^d)}\, \E^{T_0 A_{f_0}(\alpha)},
	\end{align*}
	where $A_{f_0}(\alpha) \coloneq  \widetilde{C}_{f_0} +\frac{C_{f_0}\log\alpha}{2\left(\log(\E+\alpha)\right)^{\mu+1}}$ depends on $f_0$ only through $\|f_0\|_{L^1}$, $\|f_0\|_{L^1_2}$ and $\|f_0\|_{L\log L}$. 
\end{lemma}

\begin{proof}
	Assume $\Hyp_{\Lambda}(M)$ is true. Since $|\eta^-|=|\eta| \sin\tfrac{\theta}{2} \leq \frac{|\eta|}{\sqrt{2}}$ by the assumption on the angular cross-section, the hypothesis implies 
	\begin{align*}
		\sup_{|\eta|\leq \sqrt{2}\Lambda} G(\eta^-)^{\frac{\mu+1}{1+\log\alpha}} |\hat{f}(\eta^-)| \leq M.
	\end{align*}
	With this uniform estimate at hand, we can bound the commutation error by
	\begin{align*}
		|\langle &G_{\Lambda}Q(f,f) - Q(f,G_{\Lambda}f), G_{\Lambda}f\rangle| \\
		&\leq 2 \beta t (\mu+1) M \int_{\R^d} \int_{\S^{d-1}} b\left(\frac{\eta}{|\eta|}\cdot\sigma\right) \, \left(1-\frac{|\eta^+|^2}{|\eta|^2} \right) \left(\log\langle\eta\rangle_{\alpha}\right)^{\mu} \\
		&\hskip15em \times  G_{\sqrt{2}\Lambda}(\eta^+) |\hat{f}(\eta^+)| \, G_{\sqrt{2}\Lambda}(\eta) |\hat{f}(\eta)| \,\D\sigma\D\eta,
	\end{align*}
	see equation \eqref{eq:ce-estimate}. By Lemma \ref{lem:ce}, this can be further bounded by
	\begin{align*}
		|\langle G_{\Lambda}Q(f,f) - Q(f,G_{\Lambda}f), G_{\Lambda}f\rangle| \leq \beta T_0 (\mu+1) c_{b,d} M \left(\|G_{\sqrt{2}\Lambda} f\|_{L^2}^2 + \left\| \left(\log\langle D_v\rangle_{\alpha}\right)^{\frac{\mu}{2}} G_{\sqrt{2}\Lambda} f\right\|_{L^2}^2 \right)
	\end{align*}
	for all $0\leq t\leq T_0$. Thus, the a priori bound from Corollary \ref{cor:gronwallbound} yields
	\begin{align}\label{eq:ce-gronwall}
	\begin{split}
	\|G_{\sqrt{2}\Lambda}f\|_{L^2}^2 &\leq \|\1_{\sqrt{2}\Lambda}(D_v) f_0\|_{L^2}^2 + 2 \left(\widetilde{C}_{f_0} + \beta T_0 (\mu+1) c_{b,d} M \right) \int_0^t \|G_{\sqrt{2}\Lambda} f\|_{L^2}^2\,\D\tau \\
		&\quad +2 \int_0^t \Bigg( \beta \left\|\left(\log\langle D_v\rangle_{\alpha}\right)^{\frac{\mu+1}{2}} G_{\sqrt{2}\Lambda} f\right\|_{L^2}^2 + \beta T_0 (\mu+1) c_{b,d} M \left\|\left(\log\langle D_v\rangle_{\alpha}\right)^{\frac{\mu}{2}} G_{\sqrt{2}\Lambda} f\right\|_{L^2}^2 \\
		&\hskip7em - \frac{C_{f_0}}{(\log(\E+\alpha))^{\mu+1}} \left\|\left(\log\langle D_v\rangle_{\alpha}\right)^{\frac{\mu+1}{2}} G_{\sqrt{2}\Lambda} f\right\|_{L^2}^2 \Bigg)\,\D\tau
	\end{split}
	\end{align}
	Choosing $\beta \leq \beta_0(\alpha)$ as defined in \eqref{eq:step1beta} ensures that the integrand in the last term on the right hand side of \eqref{eq:ce-gronwall} is negative. Indeed, setting $B= T_0(\mu+1)c_{b,d}M$ and $C=\frac{C_{f_0}}{(\log(\E+\alpha))^{\mu+1}}$, so that $\beta \leq \frac{C\log\alpha}{\log\alpha+2B}$, one sees that 
	\begin{align*}
		\beta \log\langle\eta\rangle_{\alpha} + \beta B - C \log\langle\eta\rangle_{\alpha} &\leq - \frac{2CB}{\log\alpha + 2B} \log\langle\eta\rangle_{\alpha} + \frac{CB\log\alpha}{\log\alpha+2B} \\&= \frac{CB \left(\log\alpha - \log(\alpha+|\eta|^2)\right)}{\log\alpha + 2B} \leq 0,
	\end{align*}
	and further, since $\log\alpha \geq \mu > 0$,
	\begin{align*}
		\beta B \leq \frac{CB\log\alpha}{\log\alpha+2B} = \frac{C\log\alpha}{2} \frac{2B}{\log\alpha+2B} \leq \frac{C\log\alpha}{2}.
	\end{align*}
	It follows that 
	\begin{align*}
		\|G_{\sqrt{2}\Lambda}f\|_{L^2}^2 &\leq \|\1_{\sqrt{2}\Lambda}(D_v) f_0\|_{L^2}^2 + 2 \left(\widetilde{C}_{f_0} + \beta T_0 (\mu+1) c_{b,d} M \right) \int_0^t \|G_{\sqrt{2}\Lambda} f\|_{L^2}^2\,\D\tau \\
		&\leq \|\1_{\sqrt{2}\Lambda}(D_v) f_0\|_{L^2}^2 + 2 A_{f_0}(\alpha) \int_0^t \|G_{\sqrt{2}\Lambda} f\|_{L^2}^2\,\D\tau.
	\end{align*}
	Now Gronwall's lemma implies 
	\begin{align*}
		\|G_{\sqrt{2}\Lambda}f\|_{L^2}^2 &\leq \|\1_{\sqrt{2}\Lambda}(D_v) f_0\|_{L^2}^2 \, \E^{2 A_{f_0}(\alpha) t} \leq \|\1_{\sqrt{2}\Lambda}(D_v) f_0\|_{L^2}^2 \, \E^{2 A_{f_0}(\alpha) T_0}.
	\end{align*}
\end{proof}

\begin{lemma}[Step 2]\label{lem:step2}
	Let $\beta, \mu>0$, $T_0 >0$, and 
	\begin{align*}
		\Lambda \geq \Lambda_0\coloneq  \frac{2\sqrt{d}}{\sqrt{2}-1}.
	\end{align*} 
	If there exist finite constants $B_1, B_2\geq 0$ such that for all $0\leq t \leq T_0$
	\begin{align*}
		\|f(t,\cdot)\|_{L^1_1(\R^d)} \leq B_1, \quad \text{and} \quad \|(G_{\sqrt{2}\Lambda} f)(t, \cdot)\|_{L^2(\R^d)} \leq B_2,
	\end{align*}
	then there exists a constant $K$ depending only on the dimension $d$ and the bounds $B_1, B_2$, such that for all $|\eta|\leq \widetilde{\Lambda}\coloneq  \frac{1+\sqrt{2}}{2} \Lambda$ and $t\in[0, T_0]$
	\begin{align*}
		|\hat{f}(t, \eta)| \leq K\, G(t,\eta)^{-\frac{2}{d+2}}.
	\end{align*}
\end{lemma}

\begin{proof}
	By Remark \ref{rem:impossible-embedding} $f$ satisfies the conditions of Lemma \ref{lem:impossible-embedding} with $\|\nabla\hat{f}\|_{L^{\infty}(\R^d)} \leq 2\pi B_1$, uniformly in $t\in[0,T_0]$. Obviously, also $\|\hat{f}\|_{L^{\infty}(\R^d)} \leq \|f\|_{L^1(\R^d)} \leq B_1$. It follows that for any $\eta\in\R^d$
	\begin{align*}
		|\hat{f}(\eta)| \leq \left(2\pi (d+2) B_1\right)^{\frac{d}{d+2}} \left( \int_{Q_{\eta}}   |\hat{f}|^2\,\D\eta \right)^{\frac{1}{d+2}}.
	\end{align*}
	where $Q_{\eta}$ is a unit cube with one corner at $\eta$, such that $\eta\cdot(\zeta-\eta) \geq 0$ for all $\zeta \in Q_{\eta}$. Since its diameter is $\sqrt{d}$, the condition $\Lambda\geq \Lambda_0$ and the choice of $\widetilde{\Lambda}$ guarantee that for $|\eta|\leq \widetilde{\Lambda}$ the cube $Q_{\eta}$ always stays inside a ball around the origin with radius $\sqrt{2}\Lambda$. 
	By the orientation of $Q_{\eta}$ and since the Fourier weight $G$ is a radial and increasing function in $\eta$, we can further estimate
	\begin{align*}
		|\hat{f}(\eta)| &\leq \left(2\pi (d+2) B_1\right)^{\frac{d}{d+2}} \, G(\eta)^{-\frac{2}{d+2}} \left( \int_{Q_{\eta}}   G(\eta)^2|\hat{f}|^2\,\D\eta \right)^{\frac{1}{d+2}} \\ 
		&\leq \left(2\pi (d+2) B_1\right)^{\frac{d}{d+2}} \, G(\eta)^{-\frac{2}{d+2}} \left\|G_{\sqrt{2}\Lambda}f\right\|_{L^2(\R^d)}^{\frac{2}{d+2}} \\
		&\leq \left(2\pi (d+2) B_1 B_2^{\frac{2}{d}}\right)^{\frac{d}{d+2}} \, G(\eta)^{-\frac{2}{d+2}}
	\end{align*}
	which is the claimed inequality with $K= \left(2\pi (d+2) B_1 B_2^{\frac{2}{d}}\right)^{\frac{d}{d+2}}$.
	\end{proof}

\begin{proof}[Proof of Theorem \ref{thm:main}]
	Let $\mu>0$ and $T_0>0$ be fixed. Set $\alpha_* = \E^{\frac{d}{2} + \frac{d+2}{2}\mu} \geq \E^{\mu}$, which is chosen in such a way that $\frac{\mu+1}{1+\log\alpha_*} = \frac{2}{d+2}$ and the function $s\mapsto \left(\log(\alpha_*+s)\right)^{\mu+1}$ is concave. 
	
	Choosing $\Lambda_0 = \frac{2\sqrt{d}}{\sqrt{2}-1}$ as in Lemma \ref{lem:step2}, we define the length scales for our induction by
	\begin{align*}
		\Lambda_N \coloneq  \frac{\Lambda_{N-1}+\sqrt{2}\Lambda_{N-1}}{2} = \frac{1+\sqrt{2}}{2} \Lambda_{N-1} = \left(\frac{1+\sqrt{2}}{2}\right)^N \Lambda_0, \quad N\in\N.
	\end{align*}
	By conservation of energy, we have 
	\begin{align*}
			\|f(t,\cdot)\|_{L^1_1} \leq \|f(t,\cdot)\|_{L^1_2} = \|f_0\|_{L^1_2} \eqcolon B_1 
	\end{align*}
	in view of Lemma \ref{lem:step2}. By Lemma \ref{lem:step1} a good (in particular uniform in $N\in\N$) choice for $B_2$ is 
	\begin{align*}
		B_2\coloneq  \|f_0\|_{L^2(\R^d)} \, \E^{T_0 A_{f_0}(\alpha_*)}.
	\end{align*}
	Define further 
	\begin{align*}
		M\coloneq  \max\left\{2 B_1 + 1, \left(2\pi (d+2) B_1 B_2^{\frac{2}{d}}\right)^{\frac{d}{d+2}} \right\},
	\end{align*}
	where the second expression is just the constant $K$ from Lemma \ref{lem:step2}.
	
	For the start of the induction, we need $\Hyp_{\Lambda_0}(M)$ to hold. Since
	\begin{align*}
		\sup_{t\in[0,T_0]} \sup_{|\eta|\leq \Lambda_0} G(\eta)^{\frac{\mu+1}{1+\log\alpha_*}} |\hat{f}(\eta)| \leq \E^{\frac{\mu+1}{1+\log\alpha_*} \beta T_0 \left(\frac{1}{2}\log\left(\alpha_* + \Lambda_0^2 \right)\right)^{\mu+1}} B_1,
	\end{align*}
	there exists $\widetilde{\beta}>0$ small enough, such that for the the above choice of $M$, $\Hyp_{\Lambda_0}(M)$ is true for all $0<\beta\leq \widetilde{\beta}$.
	
	For the induction step, assume that $\Hyp_{\Lambda_N}(M)$ is true. Setting
	\begin{align*}
		\beta = \min\{\beta_0(\alpha_*), \widetilde{\beta}\}
	\end{align*}  
	with $\beta_0(\alpha)$ from Lemma \ref{lem:step1}, all the assumptions of Lemma \ref{lem:step1} are fulfilled and it follows that 
	\begin{align*}
		\|G_{\sqrt{2}\Lambda_N} f\|_{L^2(\R^d)} \leq \|\1_{\sqrt{2}\Lambda_N}(D_v)f\|_{L^2(\R^d)} \, \E^{T_0 A_{f_0}(\alpha_*)} \leq B_2.
	\end{align*}
	Notice that the right hand side of this inequality \emph{does not depend on $M$}. Lemma \ref{lem:step2} now implies that for all $|\eta|\leq \widetilde{\Lambda_N} = \Lambda_{N+1}$ 
	\begin{align*}
		G(t,\eta)^{\frac{2}{d+2}} |\hat{f}(t,\eta)| \leq K \leq M
	\end{align*}
	for all $t\in[0,T_0]$. By the choice of $\alpha_*$ this means that $\Hyp_{\Lambda_{N+1}}(M)$ is true. 
	
	By induction, it follows that $\Hyp_{\Lambda_N}(M)$ holds for all $N\in\N$, in particular
	\begin{align*}
		\sup_{t\in[0,T_0]} \sup_{\eta\in\R^d} \E^{\beta t \left(\log\langle \eta \rangle_{\alpha_*}\right)^{\mu+1}} |\hat{f}(\eta)| \leq M.
	\end{align*}
	Another application of Lemma \ref{lem:step1} implies 
	\begin{align*}
		\|G_{\sqrt{2}\Lambda_N} f\|_{L^2(\R^d)} \leq \|f_0\|_{L^2(\R^d)}\,\E^{T_0 A_{f_0}(\alpha_*)} \quad \text{for all } N\in\N. 
	\end{align*}
	Passing to the limit $N\to\infty$, it follows that $\|G f\|_{L^2(\R^d)}\leq B$, that is, 
	\begin{align*}
		\E^{\beta t \left(\log\langle D_v \rangle_{\alpha_*}\right)^{\mu+1}} f(t, \cdot) \in L^2(\R^d).
	\end{align*}
\end{proof}

\section{Smoothing effect for arbitrary physical initial data}

\begin{proof}[Proof of Theorem \ref{thm:main}]
Let $T>0$ be arbitrary (but finite). By the already known $H^{\infty}$ smoothing property of the homogeneous Boltzmann equation for Maxwellian molecules with Debye-Yukawa type interaction, see \cite{MUXY09}, for any $0<t_0<T$ one has 
\begin{align*}
	f \in L^{\infty}([t_0, T]; H^{\infty}(\R^d)),
\end{align*}
in particular $f(t,\cdot)\in L^2(\R^d)$ for all $t_0\leq t \leq T$. 
Using $f(t_0,\cdot)\in L^1_2\cap L\log L \cap L^2$ as new initial datum, Theorem \ref{thm:mainL2} implies that there exist $\beta,M>0$ such that 
\begin{align*}
	\E^{\beta t \left(\log\langle D_v \rangle_{\alpha_*}\right)^{\mu+1}} f(t, \cdot) \in L^2(\R^d)\\ \intertext{and}
	\|\E^{\beta t \left(\log\langle \cdot \rangle_{\alpha_*}\right)^{\mu+1}} \hat{f}(t, \cdot)\|_{L^{\infty}(\R^d)} \leq M
\end{align*}
for all $t\in[t_0,T]$. By the characterisation of the spaces $\mathcal{A}^{\mu}$ (see Appendix \ref{app:analyticity}), and since $t_0$ and $T$ are arbitrary, it follows that $f(t,\cdot)\in\mathcal{A}^{\mu}(\R^d)$ for all $t>0$.
\end{proof}

\begin{proof}[Sketch of the Proof of Corollary \ref{cor:main}] 
In the notation of \cite{DFT09}, basically, the only thing that needs to be checked is that the function $\psi_{\alpha}: [0,\infty) \to [0, \infty)$, $r\mapsto \psi_{\alpha}(r) \coloneq  (\log\sqrt{\alpha+r})^{\mu+1}$ satisfies 
\begin{enumerate}[label=(\roman*)]
	\item $\psi_{\alpha}(r) \to \infty$ for $r\to \infty$
	\item $\psi_{\alpha}(r)\leq r$ for $r$ large enough
	\item there exists $R\geq 1$ such that for all $0\leq \lambda \leq 1$
	\begin{align*}
		\psi_{\alpha}(\lambda^2|\eta|^2) \geq \lambda^2 \psi_{\alpha}(|\eta|^2) \quad \text{whenever} \quad \lambda |\eta| \geq R.
	\end{align*} 
\end{enumerate}
Property (iii) is fulfilled by any concave function $\psi$ with $\psi(0) \geq 0$. This clearly is the case for $\psi_{\alpha}$ if $\alpha\geq \E^{\mu}$, see Lemma \ref{lem:subadditivity}. 

So we take the $\alpha$ from Theorem \ref{thm:mainL2} and conclude propagation with Theorem 1.2 from \cite{DFT09}.
\end{proof}

\appendix

\section{Coercivity of the Boltzmann collision operator with Debye-Yukawa Potential}\label{app:coercivity}

Since we need to take care of the dependence of the constants within our inductive approach, we present a slightly modified version of the coercivity estimate first proved by \textsc{Morimoto} \etal \cite{MUXY09}, based upon the ideas of \textsc{Alexandre} \etal \cite{ADVW00}.
\begin{proposition}[Coercivity Estimate]\label{prop:subelliptic}
Let $g\geq 0$, $g\in L^1_1(\R^d) \cap L\log L(\R^d)$. Then there exists a positive constant $C_g$ depending only on the dimension $d$, the collision kernel $b$, $\|g\|_{L^1_1}$ and $\|g\|_{L\log L}$ and constants $C>0$, $R\geq \sqrt{\E}$, depending only on the dimension $d$ and on the collision kernel $b$, such that for all $\alpha\geq 0$ and all $0\leq f\in H^1(\R^d)$ one has
\begin{align*}
	-\langle Q(g,f),f \rangle \geq \frac{C_g}{\left(\log( \alpha + \E )\right)^{\mu+1}} \left\|\left(\log\langle D_v\rangle_{\alpha}\right)^{\frac{\mu+1}{2}} f \right\|_{L^2}^2 - \left(C_g \left(\log R\right)^{\mu+1} + C \|g\|_{L^1}\right) \|f\|_{L^2}^2.
\end{align*}
\end{proposition}
\begin{remark}
  Of course, the above lower bound holds for a much larger class of functions, essentially, $ \big\|\left(\log\langle D_v\rangle_{\alpha}\right)^{\frac{\mu+1}{2}} f \big\|_{L^2} $ should be finite. 	
\end{remark}

As a first step in the proof of Proposition \ref{prop:subelliptic}
\begin{lemma}\label{lem:subelliptic}
	Assume that the angular collision kernel $b$ satisfies \eqref{eq:singularity} and \eqref{eq:momentumtransfer} and let $g\geq 0$, $g\in L^1_1\cap L\log L$. Then there exists a constant $C_g'>0$, depending only on $b$, the dimension $d$, and $\|g\|_{L^1}$, $\|g\|_{L^1_1}$, and $\|g\|_{L\log L}$, as well as a constant $R\geq \sqrt{\E}$ depending only on $d$ and $b$, such that 
	\begin{align*}
		\int_{\S^{d-2}} b\left(\frac{\eta}{|\eta|}\cdot\sigma\right)\, &\left( \hat{g}(0) - |\hat{g}(\eta^-)| \right)\,\D\sigma \geq C_g' \left( \frac{\log\langle\eta\rangle_{\alpha}}{\log(\alpha+\E)} \right)^{\mu+1} \, \1_{\{|\eta|\geq R\}}.
	\end{align*}
\end{lemma}

\begin{remark} The constant $C_g$ (respectively $C_g'$) is an increasing function of $\|g\|_{L^1}$, $\|g\|_{L^1_1}^{-1}$ and $\|g\|_{L\log L}^{-1}$, see the proof of Lemma 3 in \cite{ADVW00}. In particular, if $g$ is a weak solution of the Cauchy problem \eqref{eq:cauchyproblem} with initial datum $g_0\in L^1_2(\R^d)\cap L\log L(\R^d)$, we have $\|g\|_{L^1}=\|g_0\|_{L^1}$, $\|g\|_{L^1_1} \leq \|g\|_{L^1_2} \leq \|g_0\|_{L^1_2}$ and $\|g\|_{L\log L} \leq \log2 \|g_0\|_{L^1} + H(g_0) + C_{\delta,d} \|g_0\|_{L^1_2}^{1-\delta}$, for small enough $\delta>0$. This implies $C_g \geq C_{g_0}$.
\end{remark}

Applying the remark to the constant $C_g'$ in Lemma \ref{lem:subelliptic}, we arrive at

\begin{corollary}\label{cor:subelliptic}
	Let $g$ be a weak solution of the Cauchy problem \eqref{eq:cauchyproblem} with initial datum $g_0\in L^1_2(\R^d)\cap L\log L(\R^d)$ and angular collision kernel $b$ satisfying \eqref{eq:singularity} and \eqref{eq:momentumtransfer}. Then the conclusion of Proposition \ref{prop:subelliptic} holds with $C_g$ and $\|g\|_{L^1}$ replaced by $C_{g_0}$ and $\|g_0\|_{L^1}$, i.e.
	\begin{align*}
	-\langle Q(g,f),f \rangle \geq \frac{C_{g_0}}{\left(\log( \alpha + \E )\right)^{\mu+1}} \left\|\left(\log\langle D_v\rangle_{\alpha}\right)^{\frac{\mu+1}{2}} f \right\|_{L^2}^2 - \left(C_{g_0} \left(\log R\right)^{\mu+1} + C \|g_0\|_{L^1}\right) \|f\|_{L^2}^2,
\end{align*}
uniformly in $t\geq 0$.
\end{corollary}

\begin{proof}[Proof of Proposition \ref{prop:subelliptic}]
	We have $\langle Q(g,f),f\rangle = \Re\langle Q(g,f),f\rangle$ and by Bobylev's identity, 
	\begin{align*}
		- \Re \langle Q(g,f), f\rangle &= \Re \int_{\R^d\times \S^{d-1}} b\left(\frac{\eta}{|\eta|}\cdot\sigma\right)  \,\left[ \overline{\hat{g}(0)\hat{f}(\eta)} - \overline{\hat{g}(\eta^-)\hat{f}(\eta^+)} \right]\, \hat{f}(\eta) \, \D{\sigma}\D{\eta} \\
		&= \frac{1}{2} \int_{\R^d\times \S^{d-1}} b\left(\frac{\eta}{|\eta|}\cdot\sigma\right) \, \left\langle \begin{pmatrix} \hat{f}(\eta) \\ \hat{f}(\eta^+)\end{pmatrix}, \begin{pmatrix}
			2\hat{g}(0) & - \hat{g}(\eta^-) \\ -\overline{\hat{g}(\eta^-)} & 0 
		\end{pmatrix} \begin{pmatrix} \hat{f}(\eta) \\ \hat{f}(\eta^+)\end{pmatrix} \right\rangle_{\C^2} \,\D\sigma \D\eta \\
		&= \frac{1}{2} \int_{\R^d\times \S^{d-1}} b\left(\frac{\eta}{|\eta|}\cdot\sigma\right) \, \left\langle \begin{pmatrix} \hat{f}(\eta) \\ \hat{f}(\eta^+)\end{pmatrix}, \begin{pmatrix}
			\hat{g}(0) & - \hat{g}(\eta^-) \\ -\overline{\hat{g}(\eta^-)} & \hat{g}(0) 
		\end{pmatrix} \begin{pmatrix} \hat{f}(\eta) \\ \hat{f}(\eta^+)\end{pmatrix} \right\rangle_{\C^2} \,\D\sigma \D\eta \\
		&\quad - \frac{1}{2} \int_{\R^d\times \S^{d-1}} b\left(\frac{\eta}{|\eta|}\cdot\sigma\right) \, \left\langle \begin{pmatrix} \hat{f}(\eta) \\ \hat{f}(\eta^+)\end{pmatrix}, \begin{pmatrix}
			-\hat{g}(0) & 0 \\ 0 & \hat{g}(0) 
		\end{pmatrix} \begin{pmatrix} \hat{f}(\eta) \\ \hat{f}(\eta^+)\end{pmatrix} \right\rangle_{\C^2} \,\D\sigma \D\eta \\
		&=: I_1 - I_2.
	\end{align*}
To estimate $I_2= \frac{1}{2} \int_{\R^d\times\S^{d-1}} b\left(\frac{\eta}{|\eta|}\cdot\sigma\right)\, \hat{g}(0) \left( |\hat{f}(\eta^+)|^2 - |\hat{f}(\eta)|^2 \right)\,\D\sigma\D\eta$, we do a change of variables $\eta^+ \to \eta$ as in \cite{ADVW00} in the first part, treating $b$ as if it were integrable, and using a limiting argument to make the calculation rigorous (this is a version of the cancellation lemma of \cite{ADVW00} on the Fourier side). We then obtain with $\hat{g}(0) = \|g\|_{L^1}$ 
\begin{align*}
	I_2 = |\S^{d-2}| \int_0^{\pi/2} \sin^{d-2}\theta\, b(\cos\theta) \left[ \frac{1}{\cos^d\frac{\theta}{2}} - 1 \right]\,\D\theta \, \|g\|_{L^1(\R^d)} \, \|f\|_{L^2(\R^d)}^2.
\end{align*}
In particular, since $\frac{1}{\cos^d\frac{\theta}{2}} - 1 = \frac{d}{8} \theta^2 + \mathcal{O}(\theta^3)$, the $\theta$-integral is finite and it follows that 
	\begin{align*}
		|I_2| \leq C \|g\|_{L^1(\R^d)} \, \|f\|_{L^2(\R^d)}^2.
	\end{align*}
For the integral $I_1$, we note that since $g\geq 0$, the matrix in $I_1$ is positive definite by Bochner's theorem and has the lowest eigenvalue $\hat{g}(0) - |\hat{g}(\eta^-)|$. Therefore,
\begin{align*}
	I_1 &\geq \frac{1}{2} \int_{\R^d\times \S^{d-1}} b\left(\frac{\eta}{|\eta|}\cdot\sigma\right) \, \left( \hat{g}(0) - |\hat{g}(\eta^-)| \right) \left( |\hat{f}(\eta)|^2 + |\hat{f}(\eta^+)|^2 \right) \,\D\sigma\D\eta\\
		&\geq \frac{1}{2} \int_{\R^d} |\hat{f}(\eta)|^2 \int_{\S^{d-1}} b\left(\frac{\eta}{|\eta|}\cdot\sigma\right) \, \left( \hat{g}(0) - |\hat{g}(\eta^-)| \right)  \,\D\sigma\D\eta
\end{align*}
and by Lemma \ref{lem:subelliptic},
\begin{align*}
	I_1 &\geq \frac{C_g'}{2} \int_{\{|\eta|\geq R\}} |\hat{f}(\eta)|^2 \left(\frac{\log\langle\eta\rangle_{\alpha}}{\log(\alpha+\E)}\right)^{\mu+1}\,\D\eta \\
	&\geq \frac{C_g'}{2(\log(\alpha +\E))^{\mu+1}} \left\|\left(\log\langle D_v\rangle_{\alpha}\right)^{\frac{\mu+1}{2}} f \right\|_{L^2}^2 - \frac{C_g'}{2} \left(\frac{\log(\alpha+R^2)}{2\log(\alpha+\E)}\right)^{\mu+1} \|f\|_{L^2} \\
	&\geq \frac{C_g'}{2(\log(\alpha +\E))^{\mu+1}} \left\|\left(\log\langle D_v\rangle_{\alpha}\right)^{\frac{\mu+1}{2}} f \right\|_{L^2}^2 - \frac{C_g'}{2} \left(\log R\right)^{\mu+1} \|f\|_{L^2}.
\end{align*}
In the last inequality we used the fact that for $R \geq \sqrt{\E}$ the function $\alpha \mapsto \frac{\log(\alpha+R^2)}{2\log(\alpha+\E)}$ is decreasing. Combining the estimates of $I_1$ and $I_2$ and setting $C_g = C_g'/2$, we arrive at the claimed sub-elliptic estimate for the Boltzmann operator with Debye-Yukawa singularity.
\end{proof}

It remains to give the 
\begin{proof}[Proof of Lemma \ref{lem:subelliptic}]
	Since $g\geq 0$, $g\in L^1_1\cap L\log L$, there exists a constant $\widetilde{C}_g>0$ such that for all $\eta\in\R^d$ 
	\begin{align*}
		\hat{g}(0) - |\hat{g}(\eta)| \geq \widetilde{C}_g \left(|\eta|^2 \land 1\right).
	\end{align*}
	It is therefore enough to bound $\int_{\S^{d-1}} b(\frac{\eta}{|\eta|}\cdot \sigma) (|\eta^-|^2 \land 1)\,\D\sigma$. Recall that $|\eta^-|^2 = \frac{|\eta|^2}{2} \left(1-\frac{\eta}{|\eta|}\cdot\sigma\right)$, and, choosing spherical coordinates with pole $\frac{\eta}{|\eta|}$ such that $\frac{\eta}{|\eta|}\cdot\sigma = \cos\theta$, we obtain
	\begin{align*}
		\int_{\S^{d-1}} b\left(\frac{\eta}{|\eta|}\cdot\sigma\right) \left(|\eta^-|^2\land 1 \right) \,\D\sigma &= |\S^{d-2}|\int_{0}^{\frac{\pi}{2}} \sin^{d-2}\theta \, b(\cos\theta) \, \left(|\eta|^2 \sin^2\frac{\theta}{2} \land 1 \right)\,\D\theta \\
		&\geq \frac{|\S^{d-2}|}{4\pi^2} \int_0^{\frac{\pi}{2}} \sin^{d-2}\theta \, b(\cos\theta) \, \left(|\eta|^2 \theta^2 \land 1 \right)\,\D\theta.
	\end{align*}
	By the assumption \eqref{eq:singularity} on the singularity for grazing collisions on $b$, there exists a $\theta_0>0$ small enough such that 
	\begin{align*}
		\frac{|\S^{d-2}|}{4\pi^2} \int_0^{\frac{\pi}{2}} \sin^{d-2}\theta \, b(\cos\theta) \, \left(|\eta|^2 \theta^2 \land 1 \right)\,\D\theta &\geq 
		\frac{\kappa}{2} \frac{|\S^{d-2}|}{4\pi^2} \int_0^{\theta_0} \theta^{-1} \left(\log\theta^{-1}\right)^{\mu}\, \left(|\eta|^2 \theta^2 \land 1 \right)\,\D\theta.
	\end{align*}
	Let $R>0$ be large enough, such that $\frac{1}{R} < \theta_0$. Then for $|\eta|\geq R$ we have 
	\begin{align*}
		\frac{\kappa}{2} \frac{|\S^{d-2}|}{4\pi^2} \int_0^{\theta_0} \theta^{-1} \left(\log\theta^{-1}\right)^{\mu}\, \left(|\eta|^2 \theta^2 \land 1 \right)\,\D\theta &\geq 
		\frac{\kappa}{2} \frac{|\S^{d-2}|}{4\pi^2} \int_{\frac{1}{|\eta|}}^{\theta_0} \theta^{-1} \left(\log\theta^{-1}\right)^{\mu}\,\D\theta \\
		&= \frac{\kappa}{2} \frac{|\S^{d-2}|}{4\pi^2} \frac{1}{\mu+1} \left[ \left(\log|\eta|\right)^{\mu+1} - \left( \log\frac{1}{\theta_0} \right)^{\mu+1} \right] \\
		&\geq C \left(\log|\eta|\right)^{\mu+1}
	\end{align*}
	for some constant $C>0$ depending only on the dimension and the collision kernel $b$. We conclude by noting that for all $|\eta|\geq \sqrt{e}$ one has
	\begin{align*}
		\log|\eta| = \frac{1}{2}\log|\eta|^2 \geq \frac{\log\langle\eta\rangle_{\alpha}}{\log(\E+\alpha)},
	\end{align*}
	since for any $\alpha\geq 0$ the function $[\E, \infty) \ni s \mapsto H(s)\coloneq  \log s - \frac{\log(\alpha+s)}{\log(\alpha+\E)}$ is non-decreasing with $H(e) = 0$. 
\end{proof}

\section{Properties of the function spaces $\mathcal{A}^{\mu}$}\label{app:analyticity}

We prove a precise correspondence between the decay in Fourier space and the growth rate of derivatives of functions in $\mathcal{A}^{\mu}$. 

\begin{theorem}
	Let $\mu>0$. Then 
	\begin{align*}
		\mathcal{A}^{\mu}(\R^d) = \bigcup_{\tau>0} \mathcal{D}\left( \E^{\tau \left(\log\langle D\rangle\right)^{\mu+1}}: L^2(\R^d) \right).
	\end{align*}
\end{theorem}

Invoking a classic theorem by Denjoy and Carleman (see, for instance, \cite{Coh68,KP02,Rud86}) one can show that the classes $\mathcal{A}^{\mu}$ for $\mu>0$ are not quasi-analytic, that is, they contain non-vanishing $\mathcal{C}^{\infty}$ functions of arbitrarily small support.

\begin{proof}
	Let $\mu>0$ be fixed and assume first that $\|\E^{\tau \left(\log\langle D\rangle\right)^{\mu+1}}f\|_{L^2}<\infty$ for some $\tau>0$. Let $\alpha\in\N_0^d$ with $|\alpha|=n$ for some $n\in\N_0$. Then
	\begin{align*}
		\|\partial^{\alpha} f\|_{L^2}^2 &= \int_{\R^d} |(2\pi\I \eta)^{\alpha} \hat{f}(\eta)|^2\,\D\eta \leq (2\pi)^{2n} \int_{\R^d} \langle\eta\rangle^{2n} |\hat{f}|^{2}\,\D\xi \\
		&= 2n (2\pi)^{2n} \int_0^{\infty} t^{2n-1} \nu_f(\{\langle\eta\rangle>t\}) \,\D{t}
	\end{align*}
	where we introduced the (finite) measure $\nu_f(\D\eta)\coloneq  |\hat{f}(\eta)|^2\,\D\eta$. Since $\langle\eta\rangle \geq 1$ for all $\eta\in\R^d$, one has 
	\begin{align*}
		\nu_f(\{\langle\eta\rangle>t\}) = \nu_f(\R^d) = \|f\|_{L^2(\R^d)}^2 \quad \text{for } t<1.
	\end{align*} 
	For $t\geq 1$ we estimate
	\begin{align*}
		\nu_f(\{\langle\eta\rangle>t\}) \leq \E^{-2\tau (\log t)^{\mu+1}} \int_{\R^d} \E^{2\tau (\log\langle \eta\rangle)^{\mu+1}} \nu_f(\D\eta) = \E^{-2\tau (\log t)^{\mu+1}} \left\| \E^{\tau (\log \langle D\rangle)^{\mu+1}}f\right\|_{L^2(\R^d)}^2,
	\end{align*}
	since $1\leq t \mapsto \E^{2\tau (\log t)^{\mu+1}}$ is increasing. It follows that 
	\begin{align}\label{eq:pre-laplace}
		\|\partial^{\alpha} f\|_{L^2}^2 &\leq (2\pi)^{2n} \|f\|_{L^2(\R^d)}^2 + 2n (2\pi)^{2n} \left\| \E^{\tau (\log \langle D\rangle)^{\mu+1}}f\right\|_{L^2(\R^d)}^2 \int_1^{\infty} t^{2n-1} \E^{-2\tau (\log t)^{\mu+1}} \,\D{t}.
	\end{align}
	To extract the required growth in $n$ from the latter integral, we essentially apply Laplace's method. Indeed, substituting the logarithm and rescaling suitably yields 
	\begin{align}\label{eq:taylor}
		\int_1^{\infty} t^{2n-1} \E^{-2\tau (\log t)^{\mu+1}} \,\D{t} = \left(\frac{n}{\tau}\right)^{1/\mu} \int_0^{\infty} \E^{2\tau^{-1/\mu} n^{1+1/\mu} (t-t^{\mu+1})}\,\D{t}.
	\end{align}
	The function $h: (0, \infty) \to \R$, $h(t) \coloneq  t-t^{\mu+1}$ is strictly concave and attains its maximum at $t_* = (\mu+1)^{-1/\mu}$. If $\mu\geq 1$, $h''$ is negative and decreasing, so by Taylor's theorem we can bound 
	\begin{align*}
		h(t) \leq h(t_*) + \frac{h''(t_*)}{2} (t-t_*)^2 \1_{\{t>t_*\}}
	\end{align*}
	and obtain with $h(t_*) = \mu (\mu+1)^{-(1+1/\mu)}$, $h''(t_*) = - \mu (\mu+1)^{1/\mu}$,
	\begin{align}\label{eq:laplace}
		\left(\frac{n}{\tau}\right)^{1/\mu} &\int_0^{\infty} \E^{2\tau^{-1/\mu} n^{1+1/\mu} (t-t^{\mu+1})}\,\D{t} \\ &\leq \left(\frac{n}{\tau}\right)^{1/\mu} \left( (\mu+1)^{-1/\mu} + \frac{\sqrt{\pi}}{2\sqrt{\mu}} \left(\frac{\tau}{\mu+1}\right)^{\frac{1}{2\mu}} n^{-\frac{\mu+1}{2\mu}}\right) \E^{2\tau^{-1/\mu} \mu (\mu+1)^{-(1+1/\mu)} n^{1+1/\mu}}.
	\end{align}
Therefore, inserting the obtained bound into \eqref{eq:pre-laplace}, there exist constants $C>0$ and $b>0$, depending on $\tau$ and $\mu$, such that 
	\begin{align}\label{eq:derivatives-app}
		\left\| \partial^{\alpha} f \right\|_{L^2} \leq C^{|\alpha|+1} \E^{b |\alpha|^{1+1/\mu}} \quad \text{for all } \alpha\in\N_0^d.
	\end{align}
For $\mu\in(0,1)$ the global bound \eqref{eq:taylor} does not hold, but, as in the proof of Laplace's method for the asymptotics of integrals, one can find a suitable $\delta>0$ such that the bound \eqref{eq:taylor} holds on $[t_*-\delta, t_*+\delta]$ and the contribution to the integral outside of this interval is of much smaller order. So the right hand side of \eqref{eq:laplace} still provides an upper bound modulo lower order terms and we conclude \eqref{eq:derivatives-app} also in this case.

For the converse assume that \eqref{eq:derivatives-app} holds. We want to show that there exists a $\tau>0$ such that $\E^{\tau(\log\langle D\rangle)^{\mu+1}} f\in L^2(\R^d)$. Using that 
\begin{align*}
	\E^{2\tau (\log\langle\eta\rangle)^{\mu+1}} = 1+ \int_1^{\langle\eta\rangle} 2\tau (\mu+1) t^{-1} (\log t)^{\mu} \E^{2\tau (\log t)^{\mu+1}} \,\D{t}
\end{align*}
one obtains 
\begin{align}\label{eq:reversebound}
	\left\|\E^{\tau (\log\langle D\rangle)^{\mu+1}} f\right\|_{L^2(\R^d)}^2 = \|f\|_{L^2(\R^d)}^2 + \int_1^{\infty} 2\tau (\mu+1) t^{-1} (\log t)^{\mu} \E^{2\tau (\log t)^{\mu+1}} \, \nu_f(\{ \langle \eta \rangle >t\})\,\D{t}.
\end{align}
Next we estimate for $t>1$ and any $n\in\N_0$, since $|\eta|^2 \geq t^2 -1$ on $\{ \langle\eta\rangle >t\}$,
\begin{align*}
	\nu_f(\{ \langle\eta\rangle >t\} \leq \frac{1}{(2\pi)^{2n} (t^2-1)^{n}} \int_{\R^d} (2\pi)^{2n} |\eta|^{2n} |\hat{f}|^2 \,\D\eta.
\end{align*}
By the multinomial theorem, we have (in the standard multi-index notation)
\begin{align*}
	|\eta|^{2n} = \left(\sum_{i=1}^d \eta_i^2\right)^n = \sum_{\alpha\in\N_0^d: |\alpha|=n} \binom{n}{\alpha} \, \eta^{2\alpha},
\end{align*}
so 
\begin{align*}
	\nu_f(\{ \langle\eta\rangle >t\} &\leq \frac{1}{(2\pi)^{2n} (t^2-1)^{n}} \sum_{|\alpha| = n} \binom{n}{\alpha} \int_{\R^d} \left| (2\pi\I\eta)^{\alpha} \hat{f}(\eta)\right|^2 \,\D\eta \\
	&= \frac{1}{(2\pi)^{2n} (t^2-1)^{n}} \sum_{|\alpha| = n} \binom{n}{\alpha} \|\partial^{\alpha} f\|_{L^2(\R^d)}^2 
	\leq \frac{d^n C^{2n+2}}{(2\pi)^{2n}} \frac{1}{(t^2-1)^n} \E^{2 b n^{1+1/\mu}}
\end{align*}
by assumption.
Since this holds for any $n\in\N_0$, we even have
\begin{align*}
	\nu_f(\{ \langle\eta\rangle >t\} &\leq \exp\left[\inf_{n\in\N_0} \left(2n\log A - n\log (t^2-1) + 2 b n^{1+1/\mu}\right)\right] \\
	&= \exp\left[2 \inf_{n\in\N_0} \left(b n^{1+1/\mu} - n\log\frac{\sqrt{t^2-1}}{A} \right)\right]
\end{align*}
where for notational convenience we set $A = \frac{C^{n+1}\sqrt{d} }{2\pi}$. If $\sqrt{t^2-1} <A$, then the infimum in the above exponent is just zero, so $\nu(\{ \langle\eta\rangle >t \}) \leq 1$ in this case. If, however, $\sqrt{t^2-1} \geq A$, we get 
\begin{align*}
	\inf_{n\in\N_0} \left(b n^{1+1/\mu} - n\log\frac{\sqrt{t^2-1}}{A} \right) \leq b n_*^{1+1/\mu} - n_*\log\frac{\sqrt{t^2-1}}{A}
\end{align*}
where $n_* = \left\lfloor \left(\frac{1}{b} \frac{\mu}{\mu+1} \log\frac{\sqrt{t^2-1}}{A} \right)^{\mu} \right\rfloor$, and $\lfloor a \rfloor$ denotes the greatest integer smaller or equal to $a\in\R$. Obviously, 
\begin{align*}
	n_* \leq \left(\frac{1}{b} \frac{\mu}{\mu+1} \log\frac{\sqrt{t^2-1}}{A} \right)^{\mu} < n_*+1,
\end{align*}
so we get the bound 
\begin{align*}
	\inf_{n\in\N_0} \left(b n^{1+1/\mu} - n\log\frac{\sqrt{t^2-1}}{A} \right) \leq - \left(\frac{\mu}{\beta}\right)^{\mu} \left( \frac{1}{\mu+1} \log \frac{\sqrt{t^2-1}}{A}\right)^{\mu+1} + \log\frac{\sqrt{t^2-1}}{A}.
\end{align*}
In particular, there exists $T_*>1$ and $\beta >0$ such that for $t>T_*$, one has
\begin{align*}
	\inf_{n\in\N_0} \left(b n^{1+1/\mu} - n\log\frac{\sqrt{t^2-1}}{A} \right) \leq - \beta (\log t)^{\mu+1}.
\end{align*}
This shows,
\begin{align*}
	\nu_f(\{\langle \eta\rangle >t \}) \leq \E^{-\beta (\log t)^{\mu+1}}
\end{align*}
 for large enough $t$, and choosing $\tau<\beta/2$ in \eqref{eq:reversebound} we get the finiteness of $\left\|\E^{\tau (\log\langle D\rangle)^{\mu+1}} f\right\|_{L^2(\R^d)}$.
\end{proof}

\bigskip
\noindent\textbf{Acknowledgements.} It is a pleasure to thank the REB program of CIRM for giving us the opportunity to start this research.
	J.-M.~B.\ was partially supported by the project SQFT ANR-12-JS01-0008-01. D.~H., T.~R., and S.~V.\ gratefully acknowledge financial support by the Deutsche Forschungsgemeinschaft (DFG) through CRC 1173. D.~H.\ also thanks the Alfried Krupp von Bohlen und Halbach Foundation for financial support.
	Furthermore, we thank the University of Toulon and the Karlsruhe Institute of Technology for their hospitality.

\bibliographystyle{aomplain}
\bibliography{boltzmann}

\vfill \noindent
\textbf{Jean-Marie Barbaroux}\\
\textsc{Aix-Marseille Universit\'e, CNRS, CPT, UMR 7332, 13288 Marseille, France}\\
et \textsc{Universit\'e de Toulon, CNRS, CPT, UMR 7332, 83957 La Garde, France}\\
\textit{E-mail}: \href{mailto:barbarou@univ-tln.fr}{barbarou@univ-tln.fr}\\
\textbf{Dirk Hundertmark}\\
\textsc{Karlsruhe Institute of Technology, Englerstra{\ss}e 2, 76131 Karlsruhe, Germany}\\
\textit{E-mail}: \href{mailto:dirk.hundertmark@kit.edu}{dirk.hundertmark@kit.edu}\\
\textbf{Tobias Ried}\\
\textsc{Karlsruhe Institute of Technology, Englerstra{\ss}e 2, 76131 Karlsruhe, Germany}\\
\textit{E-mail}: \href{mailto:tobias.ried@kit.edu}{tobias.ried@kit.edu}\\
\textbf{Semjon Vugalter}\\
\textsc{Karlsruhe Institute of Technology, Englerstra{\ss}e 2, 76131 Karlsruhe, Germany}\\
\textit{E-mail}: \href{mailto:semjon.wugalter@kit.edu}{semjon.wugalter@kit.edu}
\end{document}